\documentclass[a4paper,10pt,fleqn]{amsart}
\pdfoutput=1

\usepackage[utf8]{inputenc}
\usepackage[T1]{fontenc}
\usepackage{amsmath}
\usepackage{amsfonts}
\usepackage{amsthm}
\usepackage{amssymb}
\usepackage{bbm}
\usepackage[colorlinks,breaklinks=true]{hyperref}
\usepackage{appendix}
\usepackage{aliascnt}
\usepackage{marginnote}

\usepackage{tikz,pgfplots,enumerate}
\pgfplotsset{compat=1.11}
\usetikzlibrary{calc,angles,quotes}
\tikzset{help lines/.style=very thin}
\tikzset{Grid/.style={help lines,color=gray!20}}

\makeatletter

\AtBeginDocument{%
  \@ifpackageloaded{refcheck}{%
    \@ifundefined{hyperref}{}{%
      \let\T@ref@orig\T@ref
      \def\T@ref#1{\T@ref@orig{#1}\wrtusdrf{#1}}%
      \let\@refstar@orig\@refstar
      \def\@refstar#1{\@refstar@orig{#1}\wrtusdrf{#1}}%
      \DeclareRobustCommand\ref{\@ifstar\@refstar\T@ref}%
    }%
  }{}%
}

\makeatother

\title[Degenerate fully nonlinear nonlocal MFG]%
{A strongly degenerate fully nonlinear mean field game with nonlocal diffusion}

\author{Indranil Chowdhury}
\address{\parbox{.8\linewidth}{{\normalfont\textbf{I.~Chowdhury}}\medskip\\
Indian Institute of Technology - Kanpur,\\
Department of Mathematics and Statistics, \\
Kalyanpur, Kanpur - 208016, India\medskip}}
\email{indranil@iitk.ac.in}

\author{Espen R.~Jakobsen}
\address{\parbox{.8\linewidth}{{\normalfont\textbf{E.~R.~Jakobsen}}\medskip\\
Institutt for matematiske fag, NTNU,\\
7491 Trondheim, Norway\medskip}}
\email{espen.jakobsen@ntnu.no}

\author{Mi\l{}osz Krupski}
\address{\parbox{.8\linewidth}{{\normalfont\textbf{M.~Krupski}}\medskip\\
Instytut Matematyczny, Uniwersytet Wroc\l{}awski,\\
pl.~Grunwaldzki 2/4, 50-384 Wroc\l{}aw, Poland\medskip\\
Prirodoslovno--matemati\v{c}ki fakultet, Sveu\v{c}ili\v{s}te u Zagrebu,\\
Horvatovac 102a, 10000 Zagreb, Croatia\medskip}}
\email{milosz.krupski@uwr.edu.pl}

\makeatletter
\@namedef{subjclassname@2020}{%
  \textup{2020} Mathematics Subject Classification}
\makeatother

\date{\today}

\let\OLDenum\enumerate
\renewcommand\enumerate{\vspace{0.25\baselineskip}\OLDenum\setlength{\itemsep}{0.5\baselineskip}}
\newenvironment{description*}%
  {\vspace{0.25\baselineskip}\begin{description}
    \setlength{\itemsep}{0.5\baselineskip}
    \setlength{\mathindent}{1.5\leftmargin}
  }
  {\end{description}}
   
\makeatletter
\let\orgdescriptionlabel\descriptionlabel
\renewcommand*{\descriptionlabel}[1]{%
 \let\orglabel\label
 \let\label\@gobble
 \phantomsection
 \edef\@currentlabel{#1}%
 \let\label\orglabel
 \orgdescriptionlabel{#1}%
}
\makeatother

\newtheorem{theorem}{Theorem}[section]
\newaliascnt{lemma}{theorem}\newtheorem{lemma}[lemma]{Lemma}\aliascntresetthe{lemma}
\newaliascnt{corollary}{theorem}\newtheorem{corollary}[corollary]{Corollary}\aliascntresetthe{corollary}
\newaliascnt{proposition}{theorem}\newtheorem{proposition}[proposition]{Proposition}\aliascntresetthe{proposition}
\newaliascnt{conjecture}{theorem}\aliascntresetthe{conjecture}
\theoremstyle{remark}
\newaliascnt{remark}{theorem}\newtheorem{remark}[remark]{Remark}\aliascntresetthe{remark}
\theoremstyle{definition}
\newaliascnt{definition}{theorem}\newtheorem{definition}[definition]{Definition}\aliascntresetthe{definition}

\def\equationautorefname~#1\null{problem~\upshape{(#1)}\null}
\def\partnautorefname~#1\null{Part\,(#1)\null}
\def\itemautorefname~#1\null{\,(#1)\null}

\newcounter{step}[theorem]
\newcounter{partn}[theorem]
\renewcommand{\thepartn}{\textit{\roman{partn}}}

\newcommand{\bulletdiam}{\raisebox{\dimexpr.5\fontcharht\font`S-.5\height}{$\diamond$}\ }
\newenvironment{proof*}{\begin{proof}\setcounter{step}{0}}{\end{proof}}
\newenvironment{step}{\refstepcounter{step}\bulletdiam
{\textit{Step \thestep}.}}
{\smallskip\par}
\newenvironment{step*}[1]{\refstepcounter{step}\bulletdiam {\textit{Step \thestep. #1}.}}
{\smallskip\par}
\newenvironment{part*}
{\refstepcounter{partn}\bulletdiam{\textit{Part {\rm (}\thepartn\,{\rm )}}.}}
{\smallskip\par}

\newcounter{parn}[remark]
\renewcommand{\theparn}{\textit{\alph{parn}}}
\newcommand*{\npar}{\refstepcounter{parn}\ifnum\value{parn}=1{\par({\theparn})\ }\else{\smallskip\par ({\theparn})\ }\fi}
\def\parnautorefname~#1\null{\,({#1})\null}

\DeclareMathOperator{\pv}{p.\!v.\!}

\newcommand{\dt} {\partial_t}
\newcommand{\KHJ}{K_{\text{\textit{HJB}}}}

\newcommand{\R} {\mathbb{R}}
\newcommand{\N} {\mathbb{N}}
\newcommand{\X} {{\R^d}}
\newcommand{\fL}{\mathcal{L}}
\newcommand{\fLf}{(-\Delta)^{\sigma}}
\newcommand{\fLs}{\fL^*}
\newcommand{\PX}{\mathcal{P}(\X)}
\newcommand{\Mb}{\mathcal{M}_b(\X)}
\newcommand{\T}{\mathcal{T}}
\newcommand{\Tb}{\overline{\T}}
\newcommand{\rtc}{\theta}
\newcommand{\bstr}{\omega}
\newcommand{\hd}{\alpha}
\newcommand{\hdd}{p}
\newcommand{\hb}{\beta} 
\newcommand{\Holder}[1]{\mathcal{C}^{#1}}
\newcommand{\Hb}[1]{\Holder{#1}_b(\X)}
\newcommand{\CPX}{C(\Tb,\PX)}
\newcommand{\LC}[1]{B(\T, \Hb{#1})}

\newcommand{\Cb}{C_b(\T\times\X)}
\newcommand{\Cbb}{C_b(\Tb\times\X)}
\newcommand{\Cu}{C(\T\times\X)}
\newcommand{\UC}{\text{\textit{UC}}}
\newcommand{\cF}{\mathfrak{f}}
\newcommand{\cG}{\mathfrak{g}}
\newcommand{\cD}{\mathcal{R}}

\newcommand{\uT}{g}
\newcommand{\unT}[1]{g_{#1}}
\renewcommand{\epsilon}{\varepsilon}
\newcommand{\fLlow}{\fL_r}
\newcommand{\fLhigh}{\fL^r}
\newcommand{\CaPX}[1]{\Holder{#1}(\Tb,\PX)}
\newcommand{\HJcD}{\mathcal{S}_{\text{\textit{HJB}}}}

\makeatletter
\def\@tocline#1#2#3#4#5#6#7{\relax
  \ifnum #1>\c@tocdepth 
  \else
    \par \addpenalty\@secpenalty\addvspace{#2}%
    \begingroup \hyphenpenalty\@M
    \@ifempty{#4}{%
      \@tempdima\csname r@tocindent\number#1\endcsname\relax
    }{%
      \@tempdima#4\relax
    }%
    \parindent\z@ \leftskip#3\relax \advance\leftskip\@tempdima\relax
    \rightskip\@pnumwidth plus4em \parfillskip-\@pnumwidth
    #5\leavevmode\hskip-\@tempdima
      \ifcase #1
       \or\or \hskip 1em \or \hskip 2em \else \hskip 3em \fi%
      #6\nobreak\relax
    \hfill\hbox to\@pnumwidth{\@tocpagenum{#7}}\par
    \nobreak
    \endgroup
  \fi}
\makeatother

\begin{document}
\begin{abstract}
There are few results on mean field game (MFG) systems
where the PDEs are either fully nonlinear or have degenerate diffusions.
This paper introduces a problem that combines both difficulties.
We prove existence and uniqueness for a strongly degenerate, fully nonlinear MFG system
by using the well-posedness theory for fully nonlinear MFGs established in our previous paper \cite{CJK}. It is the first such application in a degenerate setting.
Our MFG involves a controlled pure jump (nonlocal) L\'evy diffusion of order less than one, and monotone, smoothing couplings.
The key difficulty is obtaining uniqueness for the corresponding degenerate, non-smooth Fokker--Plank equation: 
since the regularity of the coefficient and the order of the diffusion are interdependent, it
holds when the order is sufficiently low. Viscosity solutions and a non-standard doubling of variables argument are used along with a bootstrapping procedure. 
\end{abstract}

\maketitle


\section{Introduction}
In this paper we show well-posedness of a strongly degenerate fully nonlinear mean field game (MFG) system, a degenerate member of the family of fully nonlinear MFGs 
introduced in \cite{CJK}, 
\begin{align}\label{eq:mfg}
  \left\{ 
  \begin{aligned} 
   &-\dt u - F(\fL u) =\cF(m)\quad&&\text{on $\T\times\X$},\\
   &\qquad u(T)=\cG(m(T))\quad&&\text{on $\X$},\\[0.2cm]
   &\dt m - \fLs(F'(\fL u)\,m)=0\quad&&\text{on $\T\times\X$},\\
   &\qquad m(0)=m_0\quad&&\text{on $\X$},
  \end{aligned}
  \right.
 \end{align}
 where $\T=(0,T)$ for a fixed $T>0$, $F'\geq 0$, and
 $\fL$ is a purely nonlocal L\'evy operator of order less than one, i.e.
    \begin{align}\label{defL}
     \fL \phi(x) = \int_\X \big(\phi(x+z)-\phi(x)\big)\,\nu(dz),
   \end{align}
   for some L\'evy measure $\nu$ (see \autoref{def:levy}). This problem is fully nonlinear and can degenerate in two ways: when $F'=0$ or when $\fL$ degenerates (no diffusion in some directions). 
 
MFGs are limits of $N$-player stochastic games as $N\to\infty$ under some symmetry and weak interaction conditions on the players. 
 Nash equilibria are characterized by a coupled system of PDEs---the MFG system, here \autoref{eq:mfg}---consisting of a backward  Hamilton--Jacobi--Bellman (HJB) equation for the value function $u$ of the generic player and a forward Fokker--Planck (FP) equation for the distribution $m$ of players. 
 A rigorous mathematical theory of such problems started with the work of Lasry--Lions~\cite{MR2269875,MR2271747,MR2295621}
 and Huang--Caines--Malham\'e~\cite{MR2346927,MR2344101} in 2006,
 and today this is a large and rapidly expanding field, mostly focused on the so-called PDE  and probabilisitc approaches.
 Extensive background and recent developments can be found in e.g.~\cite{MR4214773,MR3134900,MR3752669,MR3753660,MR3559742,MR3967062,
 MR2762362}
 and the references therein.

In contrast to the more classical setting,
 in \cite{CJK} we introduced a class of MFGs including \autoref{eq:mfg}, where not only the drift  is controlled, but also the diffusion.
 To be more precise, the players control the time change rate of their individual (L\'evy) noises, which for self-similar noise processes like the Brownian motion or the $\alpha$-stable process is equivalent to a classical controlled
 diffusion \cite{MR2179357,MR2322248}. See \cite[Section 3]{CJK} for more details and a semi-heuristic derivation of \autoref{eq:mfg}. 
Controlled diffusions \cite{MR2179357} are key ingredients of portfolio optimization in finance~\cite{MR2178045, MR2380957, MR2533355}. Despite many applications of MFGs e.g. in economics \cite{MR2762362,MR3363751}, see also \cite{MR0172689}, controlled diffusions is a rare and novel subject in the context of MFGs, so far mostly addressed by probabilistic methods:  \cite{MR3332857,BT22} shows existence 
for 
local MFG systems with and without common noise, MFGs of controls are considered in \cite{djete2023mean}, and~\cite{MR4158808,MR3980873} consider problems perturbed by bounded nonlocal operators.
Some results by PDE methods can be found in \cite{Ricciardi}, as well as \cite{MR4361908, MR4702626}
for uniformly elliptic (stationary second order) problems. 

In \cite{CJK} we developed an abstract existence and uniqueness theory for \autoref{eq:mfg} and extensions involving also controlled drift, and applied it to prove well-posedness of non-degenerate MFGs with local or nonlocal diffusion, i.e.~involving $F'\geq\kappa>0$ and a non-degenerate L\'evy operator $\fL$. 
Typical examples are the Laplacian $\fL=\Delta$ and the fractional Laplacian $\fL=-(-\Delta)^\alpha$ for $\alpha\in(0,2)$. 

 The main objective of this paper is to verify the necessary conditions for well-posedness established in \cite{CJK}
 in the first known example of a degenerate MFG with controlled diffusion, namely the nonlocal MFG system \eqref{eq:mfg}. Nonlocal MFGs have been studied in \cite{MR3912635,MR3934106,MR4309434,MR4223351} in the case of non-degenerate and uncontrolled diffusion, while for degenerate diffusions we only know of the results of~\cite{MR3399179} where uncontrolled diffusions are considered.  
  
 Observe that
 with $(f,\uT) = \big(\cF(m),\cG(m(T))\big)$, the first pair of equations in \autoref{eq:mfg}
 form a terminal value problem for a fully nonlinear HJB equation,
 \begin{align}\label{eq:hjb}
 \left\{
  \begin{aligned}
   &-\dt u - F(\fL u) = f\quad&&\text{on $\T\times\X$},\\
   &\qquad u(T)= \uT\quad&&\text{on $\X$}.
  \end{aligned}
  \right.
 \end{align}
 In this case the viscosity solution framework applies, but we consider classical solutions so that $\fL u$  
 is well-defined pointwise. Since we cannot expect any regularizing effect, we rely on the comparison principle to transfer the regularity of $f$ and $\uT$ onto the
 solution $u$. With $b=F'(\fL u)$
 the second pair of equations in \autoref{eq:mfg}
 form an initial value problem for a FP equation,
 \begin{align}\label{eq:fp}
 \left\{
  \begin{aligned}
  & \dt m - \fLs(bm)=0\quad&&\text{on $\T\times\X$},\\ 
   &\qquad m(0)=m_0\quad&&\text{on $\X$}.
  \end{aligned}
  \right.
 \end{align}
  Since $b$ need not be very regular and may even degenerate, we consider very weak (measure-valued) solutions of \autoref{eq:fp}.

Uniqueness for such FP equations were not previously known and its proof constitutes a large portion of our work. We use a Holmgren-type
argument,
 which requires the construction of a suitable test function 
solving a strongly degenerate dual equation with a coefficient of low regularity. This construction
relies on viscosity solutions theory, non-standard doubling of variables arguments, and bootstrapping to get optimal results. Because the regularity of the coefficient decreases with increasing order of $\fL$, our argument only works when the order of the operator is low enough (see \autoref{rem:gamma}).

The paper is structured as follows. \autoref{section:main} contains the main results of the paper: existence and uniqueness for a degenerate model of mean field games with controlled diffusion and uniqueness of solutions of the associated FP equation.
In \autoref{section:stochastics} we recall the model of controlling the diffusion (time change rate of the L\'evy process) and show how different control problems translate into various types of HJB equations; then we give a natural example when the MFG system \eqref{eq:mfg} may be degenerate. \autoref{section:preliminaries} contains preliminary results needed in the paper, regarding L\'evy operators and viscosity solutions.
In \autoref{section:hjb-fp} we first recall the general well-posedness theory of \cite{CJK}. To conclude (by applying these results), suitable 
properties of solutions are then established separately for HJB and FP equations.
Finally, the fundamental results on existence and regularity of solutions of the dual equation, leading to uniqueness of solutions of the FP equation, are presented in \autoref{section:dual}.

\section{Main results}\label{section:main}
  By $B_r$ and $B_r(x)$ we denote the balls centred at $0$ and $x$. Let $\PX$ consist of probability measures on $\X$, a~subspace of
 the space of bounded Radon measures $\Mb=C_0(\X)^*$.
  We equip $\PX$ with the topology of weak (narrow, vague) convergence of measures (see \cite{CJK} for details).
 \begin{definition}\label{def:levy}
 A non-negative Radon measure on $\X$ satisfying
 \begin{equation*}
 \nu\big(\{0\}\big)=0,\qquad\int_\X \big(1\wedge |z|\big)\,\nu(dz)<\infty
 \end{equation*}
 is called a L\'evy measure (of order less than one).
 A L\'evy measure $\nu$ is symmetric at the origin if 
   $\nu(A)=\nu(-A)$ for every $A\subset B_1$.
 \end{definition}
 \begin{definition}\label{def:holder}
  A function $\phi$ is H\"older-continuous at $x\in \X$ with parameter $\hd\in(0,1]$ if
  for some $r>0$
  \begin{align}\label{eq:holder}
   [\phi]_{\Holder{\hd}(B_r(x))} = \sup_{y\in B_r(x)\setminus\{x\}} \frac{|\phi(x)-\phi(y)|}{|x-y|^\hd}<\infty.
  \end{align}
  The space $\Holder{\hd}(\X)$ consists of functions which are H\"older-continuous
  at every point in $\X$ with parameter $\hd$.
  Further, define $\Hb{\hd} = \{\phi:\|\phi\|_{\hd}<\infty\}$, where
  \begin{align*}
      [\phi]_{\hd} = \sup_{x\in \X}\,[\phi]_{\Holder{\hd}(B_1(x))} 
      \quad \text{and}\quad \|\phi\|_{\hd} = \|\phi\|_{L^\infty(\X)} + [\phi]_{\hd}.
  \end{align*}
  \end{definition}
  Note that the definition of $\Hb{\hd}$ is equivalent 
 to the more standard notation, where the supremum in \eqref{eq:holder} is taken over $|x-y|\in\X\setminus\{0\}$.
 The space $\Hb{1}$ consists of bounded, Lipschitz-continuous functions.
 By $C^1(\X)$, $C^2(\X)$ we denote spaces of once or twice continuously differentiable functions.
 
  \begin{definition}\label{def:bounded}
   When $X$ is a normed space, $B(\T,X)$ denotes the space of bound\-ed functions from $\T$ to $X$, i.e.~$B(\T,X) = \big\{u:\T\to X\ :\ \textstyle\sup_{t\in\T}\|u(t)\|_X<\infty\big\}$.
  \end{definition} 
In \autoref{eq:mfg}, we use the following assumptions:
 \begin{description*}
  \item[(L)\label{L:deg}]   $\fL$ is given by \eqref{defL} for $2\sigma\in(0,1)$ and L\'evy measure $\nu$ satisfying
 \begin{align*}
      \int_{B_1} \Big(1\wedge\frac{|z|^\hdd}{r^\hdd}\Big)\,\nu(dz) 
      \leq \frac{K}{\hdd-2\sigma}r^{-2\sigma}
 \end{align*}
 for a constant $K\geq 0$ and every $\hdd \in (2\sigma, 1]$, $r\in(0,1)$.
 \end{description*}
 \medskip
  \begin{description*}
  \item[(A1)\label{a:F1}]
   $F\in C^1(\R)$, $F'\in\Holder{\gamma}(\R)$ for $\gamma \in (0,1]$, and $F'\geq 0$;
  \item[(A2)\label{a:F2}] $F$ is convex;     
  \item[(A3)\label{a:m}] 
     $m_0\in\PX$;
  \end{description*}
  \begin{description*}
  \item[(A4)\label{a:fg1}] 
   $\cF:\CPX\to\Cb$, \quad $\cG:\PX\to C_b(\X)$
   are continuous;
  \item[(A5)\label{a:fg2}]
   $\cF$ and $\cG$ are monotone operators.
   \end{description*}
  \medskip
  \begin{description*}
   \item[(R)\label{D:deg}] There are $\alpha \in (2\sigma, 1]$ and $M\in[0,\infty)$ such that the range
 \begin{align*}
 \cD=
 \big\{\big(\cF(m),\cG(m(T))\big) : \ m \in \CPX \big\}
 \end{align*}
 satisfies $\cD\subset \cD_0(\alpha,M)$, where\footnote{See \autoref{def:holder}, \autoref{def:bounded} of spaces $\Hb{\hd}$ and $B(\T,X)$; \textit{UC} = uniformly continuous.}
 \begin{align*}
    \cD_0(\alpha,M)=\Big\{(f,g) : \ \ (i)\ &\  f\in\UC(\T\times\X)\cap \LC{\hd}, \\
 (ii)\ & \ g\in \Hb{\hd},\\ (iii)\ & \ \sup_{t\in\T}\|f(t)\|_\hd+\|\uT\|_\hd\leq M \Big\}.
\end{align*}
  \end{description*}

\begin{remark}When $\nu$ is absolutely continuous with respect to the Lebesgue measure, \ref{L:deg} is equivalent to the \emph{upper bound} $\frac{d\nu}{dz}\leq C\frac{1}{|z|^{d+2\sigma}}$ for $|z|<1$, and hence is satisfied for the fractional Laplacian $\Delta^{\sigma}$ \cite{MR3156646}, the nonsymmetric nonlocal operators used in finance 
\cite{MR2042661},
and a large class of non-degenerate and degenerate operators.
Any bounded L\'evy measure (``$\sigma\!\downarrow\!0$\,''), and hence any bounded nonlocal (L\'evy) operator, is also included.
Note that there is no further restriction on the tail  of $\nu$ (the $B_1^c$-part) and hence no explicit moment assumption on the L\'evy process and the solution of the FP equation $m$.
See \cite{CJK} for more details.
\end{remark}
\begin{definition}
We say that $(u,m)$ is a classical--very weak solution of \autoref{eq:mfg} if \mbox{$u\in \Cbb$} solves the HJB part in the classical sense\footnote{i.e.~$u$ is a pointwise solution such that $\fL u$ and $u_t$ are continuous.} and $m\in \CPX$ solves the FP part in the sense of distributions. See \cite{CJK} for the precise definition.
\end{definition}
 \begin{theorem}\label{thm:deg}
Assume~\ref{D:deg},~\ref{L:deg},~\ref{a:F1},~\ref{a:m}.
 If in addition
  \begin{enumerate}
      \item \label{item:cor-degen-ex} \ref{a:fg1} holds,
      then there exists a classical--very weak  solution of \autoref{eq:mfg};
      \item \label{item:cor-degen-u1} \ref{a:F2},~\ref{a:fg2} hold and $\frac{2\sigma}{(\hd-2\sigma)}\big(1+\frac1{1-2\sigma}\big)<\gamma$,
      then \autoref{eq:mfg} has at most one classical--very weak solution;
      \item \label{item:cor-degen-u2}  \ref{a:F2},~\ref{a:fg2} hold, $\nu$ is symmetric at the origin, 
      and $\frac{2\sigma}{(\hd-2\sigma)}\big(1+\frac1{1-\sigma}\big)<\gamma$, then \autoref{eq:mfg} has at most one classical--very weak solution.
  \end{enumerate}
  \end{theorem}
  \begin{remark}\label{rem:gamma}
 \npar When $\gamma=\hd=1$, the condition in \autoref{thm:deg}\autoref{item:cor-degen-u1} becomes
 \mbox{$8\sigma (1-\sigma)<1$},
 which leads to $2\sigma <\frac{2-\sqrt{2}}{2}\approx \frac{3}{10}$.
 In~\!\autoref{item:cor-degen-u2} we obtain 
 \mbox{$\sigma(7-4\sigma)<1$}
 and then $2\sigma<\frac{7-\sqrt{33}}{4}\approx\frac{4}{13}$.
  \npar See \autoref{section:stochastics} for an example involving the fractional Laplacian and a strongly degenerate power type nonlinearity $F$.
 \end{remark}
 
 \autoref{thm:deg} is proved in \autoref{section:hjb-fp}. We show that it follows from the more general existence and uniqueness theory of \cite{CJK}. The main result is recalled in  \autoref{thm:main} and relies on additional assumptions \ref{R1}--\ref{U2}. Conditions \ref{R1}--\ref{U1} relate to existence, regularity, and stability of solutions of the HJB equation~\eqref{eq:hjb}, and essentially follow from the comparison principle for viscosity solutions
 (see \autoref{thm:hjb-viscosity}).
 These results are gathered in \autoref{thm:hjb-lipschitz} and \autoref{cor:hjb-l1-s}.
 The hard part of this paper is to verify~\ref{U2}, a uniqueness condition for the FP equation~\eqref{eq:fp}---see \autoref{thm:fp-uniqueness} and \autoref{cor:b-uniq}. Let us now discuss this result in more detail.
 
 \subsection*{Degenerate and non-Lipschitz FP equations}
 \hyperref[eq:fp]{Problem~\eqref{eq:fp}} is the FP or forward Kolmogorov equation for the SDE
 \begin{equation}\label{eq:fp-sde}
  dZ(t) = b(t,Z(t))\,dX(t),\qquad Z(0) = Z_0\sim m_0,
 \end{equation}
 where $X$ is the L\'evy process with infinitesimal generator $\fL$. We refer e.g.~to~
 \cite{MR3443169}
for a survey of classical results when $X$ is a Brownian Motion (with drift) and $\fL$ is local with triplet  $(c,a,0)$.
 For general L\'evy processes, we mention the recent results of \cite{MR3771750} for $b$ independent of $t$
 and \cite{MR3949969,MR4168386} for connections to the FP equation.
  
 For local degenerate equations with locally Lipschitz or $W^{2,1}$ coefficients, uniqueness has been proved by Holmgren,
 probabilistic or variational methods~\cite{MR3391701,MR2375067,MR2607035,MR2450159}.
 For degenerate problems with either rougher coefficients or nonlocal diffusions, we are not aware of any prior results.
 Note that if $\fL = c\cdot\nabla$ and $b$ is continuous but not Lipschitz,
 then \autoref{eq:fp} does not have a unique solution~\cite{MR3391701}, and~\cite{MR0164154} shows limitations for the Holmgren method. 
 
 Let us state the subsequent condition.
 \begin{description*}
  \item[(B)\label{a1'':b}]$b\in\Cu$ and $b(t,x)\in[0,B]$
  for fixed $B\in[0,\infty)$ and every $(t,x)\in\T\times\X$;
in addition,
  $b\in\LC{\hb}$ for some $\hb>0$.
 \end{description*} 
 We prove the following result.
  \begin{theorem}\label{thm:fp-uniqueness}
  Assume~\ref{L:deg}, \ref{a:m}, and  \ref{a1'':b}.
  If either\footnote{If $2\sigma\in(0,\frac{3-\sqrt5}2)$, then 
  $2\sigma+ \frac{2\sigma}{1-2\sigma}\in(0,1)$,
  and if $2\sigma\in\big(0,\frac{5-\sqrt{17}}{2}\big)$,
  then $2\sigma+ \frac{2\sigma}{1-\sigma}\in(0,1)$,
  thus in both cases $\beta$ is chosen from a non-empty interval (cf.~\autoref{rem:optimal1}).}
  \begin{enumerate}
      \item\label{item:fp-uniq-deg} $2\sigma\in\big(0,\frac{3-\sqrt{5}}{2}\big)$ and $\hb\in \big(2\sigma+\tfrac{2\sigma}{1-2\sigma},1\big]$; or
      \item\label{item:fp-uniq-deg1} $\nu$ is symmetric at the origin,
       $2\sigma\in\big(0,\frac{5-\sqrt{17}}{2}\big)$,
       and $\hb\in \big(2\sigma+\tfrac{2\sigma}{1-\sigma},1\big]$;
  \end{enumerate}
  then \autoref{eq:fp}
  has precisely one very weak solution.
 \end{theorem} 
  
From this result condition \ref{U2} follows.
Note that $b$ can vanish and does 
  not have to be locally Lipschitz (or have weak second derivatives).
  Existence of solutions is established in \cite[Theorem 6.6]{CJK}. The uniqueness proof is original and follows from a Holmgren-type argument\footnote{Holmgren's method is based on a duality argument and goes back the early 20th century and linear PDEs~\cite[Chapter 5]{MR1301779}.
 Such arguments are widely used in the modern theory of PDEs, including nonlinear equations, see e.g.~\cite{MR2286292}.} 
 which uses a suitable test function
 solving a strongly degenerate dual equation. The construction of such a test function is new and relies on the nonlocal nature of the problem. It is based on viscosity solutions techniques and a bootstrapping argument. The reason we are able to treat degenerate and non-Lipschitz problems, is our non-standard adaptation of the doubling of variables argument in \autoref{section:dual}.

\section{Controlled time change rates and examples of degenerate MFGs}
\label{section:stochastics}
\subsection{The control problem and the HJB equation}
 We briefly recall the model introduced in \cite{CJK}.
 Let $X_t$ be a L\'evy process (with infinitesimal generator given by~\ref{L:deg}), 
  generating the filtration $\mathcal{F}_t$.
 For $(t,x)\in\T\times\X$ and $s\geq t$,
 let $X_s^{t,x}= x+ X_s-X_t$.
 Then, for an absolutely continuous random time change $\rtc_s$ such that $\rtc_t=t$,
 $\rtc_s'$~is deterministic at $s=t$, and $\rtc_{s+h}-\rtc_{s}$ is independent of $\mathcal{F}_{\rtc_s}$
 for all $s,h\geq0$, we define a time-changed process 
     $Y^{t,x,\rtc}_s = X_{\rtc_s}^{t,x}$.
 It is an inhomogeneous Markov process.
To control $Y_s^{t,x,\rtc}$, we introduce a running gain (profit, utility) $\ell$, a
terminal gain $\uT$, and an expected total gain functional 
 \begin{equation*}
  J(t,x,\rtc) = E\bigg(\int_t^T
  \ell\big(s,Y_s^{t,x,\rtc},\rtc_s'\big)\,ds+\uT\big(Y_T^{t,x,\rtc}\big)\bigg).
 \end{equation*}
We consider the corresponding value function $u$ (the optimal value of $J$), given by
 \begin{equation}\label{eq:JJ}
  u(t,x) = \sup_{\rtc\in\mathcal{A}} J(t,x,\rtc),
 \end{equation}
where $\mathcal{A}$ is a suitable set of admissible controls.
Assume $\ell(t,x,\zeta) = -L(\zeta)+f(t,x)$.
Using the dynamic programming principle, we then obtain the following Bellman equation (see \cite{CJK} for more details)
\begin{equation*}
  -\dt u = \sup_{\zeta\geq0}\Big(\zeta \fL u -L(\zeta)\Big)  + f(t,x).
  \end{equation*}
satisfied e.g.~in the viscosity sense (see \autoref{section:preliminaries}),
where $\zeta$  denotes  the (deterministic) value of $\rtc'_t$ to simplify the notation.
  The Bellman equation can be expressed in terms
  of the Legendre--Fenchel transform $F(z)=\sup_{\zeta\geq0}\big(\zeta z-L(\zeta)\big)$ as
  \begin{equation*}
  -\dt u = F\big(\fL u\big) + f(t,x).
  \end{equation*}

  For convenience we provide some prototypical examples of cost functions~$L$ that can (or cannot) be used to derive \autoref{eq:mfg}, and their corresponding Legendre--Fenchel transforms.
  Recall the non-degeneracy assumption used in \cite{CJK}.
  \begin{description*}
 \item[(A1$^{\prime}$)\label{F:ndeg}] \ref{a:F1} holds and $F'\geq\kappa$ for some $\kappa>0$ (i.e.~$F$ is strictly increasing).
 \end{description*}
 \begin{table}[ht]
  $$\arraycolsep=0.15in\def\arraystretch{1.25}\begin{array}{c|c|c}
      &L:[0,\infty)\to\R\cup\{\infty\} &  F:\R\to\R \\\hline
      (a)&
     {\text{\large $\chi$}}_{\{ \kappa\}}(\zeta) & \kappa z\\
      (b)&{\text{\large $\chi$}}_{[0,\kappa]}(\zeta)     & \kappa z^+\\
      (c)&\big({\text{\large $\chi$}}_{[0,\kappa]}(\zeta)+\epsilon\big)(\frac{1}{\kappa}\zeta^2-\zeta) &
      \mathbbm{1}_{[-\epsilon,\epsilon)}(z)\frac{\kappa}{4\epsilon}(z+\epsilon)^2+\mathbbm{1}_{[\epsilon,\infty)}(z)\kappa z\\
      (d)&\frac{1}{q}\zeta^q & \frac{q-1}{q}(z^+)^{\frac{q}{q-1}} \\
      (e)&\zeta\log(\zeta)-\zeta & e^z\\
      (f)&\big(\chi_{[\kappa,\infty)}(\zeta)+1\big)L_0(
      \zeta-\kappa)& F_0(z)+\kappa z
  \end{array}\vspace{0.5\baselineskip}$$
 \caption{Pairs of Legendre--Fenchel conjugate functions (see \mbox{\autoref{rem:table}}).
 Here ${\text{\large $\chi$}}_A(x) = \infty$ for $x\not\in A$,
 ${\text{\large $\chi$}}_A(x) = 0$ for $x\in A$; and $z^+ = \max(z,0)$; $(L_0,F_0)$ is an arbitrary conjugate pair.}\label{tab:L-transform}
 \end{table}\vspace{-\baselineskip}
  \begin{remark}[On \autoref{tab:L-transform}]\label{rem:table}
 \npar Players are forced to always choose the same control,~$\kappa$.
 The MFG system reduces to a pair of linear heat equations.
 \npar Players can choose a control between $0$ and $\kappa$, for a constant (zero) cost.
 Hamiltonian $F$ is not differentiable, \ref{a:F1} fails.
 \npar Players can choose a control between $0$ and $\kappa$, and the cost $L$ is strongly convex on $[0,\kappa]$.
 Hamiltonian $F\in C^1(\R)$, with $F'$ Lipschitz-continuous, is sufficiently regular, while not strictly convex (note how this can be used to approximate the previous case).
 Conditions~\ref{a:F1}, \ref{a:F2} are satisfied, but not \ref{F:ndeg}.
 \npar The standard linear-``quadratic'' control in the case of the fractional Laplacian (see the example below).
 \npar A cost function resulting in $F\in C^\infty(\R)$; $F>0$, but \ref{F:ndeg} is not satisfied.
 \npar This template can be used to modify any conjugate pair $(L_0,F_0)$ in order to adjust the lower bound on the derivative, e.g.~to satisfy condition \ref{F:ndeg}.
 \end{remark}
 
 \subsection{Example of a strongly degenerate fully-nonlinear MFG}
Let 
$L(\zeta)=\frac{1}{q}\zeta^q$ for $q>1$, and assume the infinitesimal generator of $X$ is the fractional Laplacian,
$\fL= -\fLf$. 
We have $F(z)= \frac{q-1}{q} (z^+)^{\frac q{q-1}}$
(cf.~\autoref{tab:L-transform} above) and $\fLs=\fL$, and hence the MFG system takes the form
\begin{align}\label{eq:mfg-ex}\left\{\begin{aligned}
 & -\dt u - \frac{q-1}{q}\Big([-\fLf u]^+\Big)^{\frac q{q-1}}=\cF(m),\\
  & \ \dt m  + \fLf \Big([-\fLf u]^+m\Big)^{\frac1{q-1}}=0.
\end{aligned}\right.\end{align}

These equations are strongly degenerate and $F$
satisfies~\ref{a:F1} with $\gamma = \tfrac 1{q-1}$, as well as \ref{a:F2}.
Existence of solutions of \autoref{eq:mfg-ex} follows from \autoref{thm:deg}\autoref{item:cor-degen-ex}
if $2\sigma \in(0,1)$, $m_0$ satisfies \ref{a:m}, $\cF, \cG$ satisfy \ref{a:fg1} and \ref{D:deg}.
 If e.g.~$\alpha=1$, $q<q_c(\sigma)=\frac{1+\sigma}{2\sigma(2-\sigma)}$, and \ref{a:fg2} holds, we also have uniqueness (see \autoref{thm:deg}\autoref{item:cor-degen-u2}).
Note that $q_c$ is decreasing, $q_c(\frac12)=1$, and $\lim\limits_{\sigma\to0^+}q_c=\infty$.

\subsection{Relation with classical continuous control} When the L\'evy process $X$ is self-similar,\footnote{$X$ is
  self-similar if there exists $c>0$ such that for all $a,t>0$, $a^c
X_{at}=X_t$ in distribution.} the control of the time change rate introduced in \cite{CJK} can be
interpreted as the classical continuous control,
i.e.~control of the size of the spatial increments of the process. Consider the optimal control problem~\eqref{eq:JJ} with new controls and process $Y_t$: Let the non-negative control processes $\lambda$ replace $\rtc'$ and controlled process $Y_t$ now be given by the SDE
 \begin{align*}
  dY_s =  \lambda_s dX_s=\lambda_s\int_\X \, z\, \tilde{N}(dt,dz), \quad \text{and} \quad Y_t=x, 
 \end{align*}
 where $\tilde{N}$ is
 the compensated Poisson measure defined from
 $X$.\footnote{$\tilde{N}(dt,dz)\!=\!
   N(dt,dz)-\mathbbm{1}_{B_1}(z)\,\nu(dz)\,dt$, where $N$ is a
   Poisson random measure with intensity~$\nu$.} 
This is a classical control problem, and 
 under suitable
assumptions it leads to the following Bellman equation (see~\cite{MR2974720,MR4047981})
\begin{align}\label{eq:HJB4}
  -\dt u=\sup_\lambda \bigg( \pv\int_\X \big( u(x+\lambda z)- u(x)\big)
  \frac{c_{d,\sigma}}{|z|^{d+2\sigma}}\,dz- \widehat{L}(\lambda) +f(s,x)\bigg),
  \end{align}
where $\pv$ denotes the \emph{principal value}.
Self-similarity (seen through $\nu$) then yields
\begin{align*}
  &\pv\int\big( u(x+\lambda z) -
  u(x)\big)\frac{c_{d,\sigma}}{|z|^{d+2\sigma}}\,dz\\
  &\qquad=\lambda^{2\sigma}\pv\int\big(u(x+z) - u(x)\big)
  \frac{c_{d,\sigma}}{|z|^{d+2\sigma}}\,dz=-\lambda^{2\sigma}\fLf u(x).
  \end{align*}
Let $\lambda^{2\sigma}=\zeta$ and $\widehat{L}(\lambda)= \frac{1}{q}\lambda^{\frac{q}{2\sigma}}= L(\zeta)$,
and $f=\cF(m)$.
Then the Bellman equations in \eqref{eq:mfg-ex} and \eqref{eq:HJB4} coincide.
This means that in this case the classical continuous control problem
and the original controlled time change rate problem coincide as well.

\section{Preliminaries}\label{section:preliminaries}  
\subsection{L\'evy measures and operators.}
 The following notion allows us to work without moment assumptions on the tail of the L\'evy measure.
 \begin{definition}
  A real function $V\in C^2(\X)$ is a Lyapunov function if $V(x) = V_0\big(\sqrt{1+|x|^2}\big)$
  for some  subadditive, non-decreasing function $V_0:[0,\infty)\to[0,\infty)$
  such that $\|V_0'\|_\infty,\|V_0''\|_\infty\leq 1$, and $\lim\limits_{x\to\infty}V_0(x)= \infty$.
 \end{definition}
 See \cite{CJK} for more details on how it can be used to characterize tightness and in turn permit any probability measure as initial data $m_0$ and more refined results. In this paper we only need the subsequent observation \cite[Corollary 4.12]{CJK}.
\begin{lemma}\label{lemma:levy-lyapunov}
  For every L\'evy operator (in particular satisfying~\ref{L:deg}), 
  there exists a Lyapunov function $V$ such that 
  \mbox{$\|\fL V\|_\infty <\infty$}.
 \end{lemma}
 Next we prove a result concerning L\'evy operators satisfying~\ref{L:deg}.
 \begin{proposition}\label{prop:lap-lip}
  Assume~\ref{L:deg} and $\phi\in \Hb{\hdd}$ for some $\hdd\in(2\sigma,1]$.
  Then
  \begin{align}\label{eq:lap-lip1}
     \|\fL \phi\|_\infty 
     \leq \frac{K}{\hdd-2\sigma} [\phi]_{\hdd} +2\|\phi\|_\infty \nu\big(B_1^c\big)
  \end{align}
  and 
  \begin{align}\label{eq:lap-lip2}
   [\fL \phi(x)]_{\hdd-2\sigma}
   \leq 2\Big(\frac{K}{\hdd-2\sigma}+\nu\big(B_1^c\big)\Big)[\phi]_{\hdd}.
  \end{align}
  Consequently, $\fL: \Hb{\hdd}\to \Hb{\hdd-2\sigma}$ is a bounded operator.
 \end{proposition}
 \begin{proof}
  Estimate~\eqref{eq:lap-lip1} is a simple consequence of~\ref{L:deg}.
  To obtain~\eqref{eq:lap-lip2}, we write
  \begin{align*}
   |\fL &\phi(x)-\fL \phi(y)|
   \leq \int_{B_1}{\big|\big(\phi(x+z)-\phi(x)\big) - \big(\phi(y+z)-\phi(y)\big)\big|}\,\nu(dz)\\ 
   &\quad+ \int_{B_1^c}\big|\big(\phi(x+z)-\phi(y+z)-\phi(x)+\phi(y)\big)\big|\,\nu(dz)
   = \mathcal{I}_1+\mathcal{I}_2.
  \end{align*}
  For $|x-y|\leq 1$ (cf.~\autoref{def:holder}, where $y\in B_1(x)$), we get
  \begin{align*}
   \mathcal{I}_1 &\leq 
   2[\phi]_{\hdd}\bigg(\int_{B_{|x-y|}} {|z|^\hdd}\,\nu(dz)
   + \int_{B_1\setminus B_{|x-y|}}{ |x-y|^{\hdd}}\,\nu(dz)\bigg) \\
   &= 2[\phi]_{\hdd}|x-y|^\hdd\int_{B_1} \bigg(1\wedge\frac{|z|^\hdd}{|x-y|^\hdd}\bigg)\,\nu(dz)
   \leq \frac{2K}{\hdd-2\sigma} [\phi]_{\hdd}|x-y|^{\hdd-2\sigma}.
  \end{align*} 
   Finally,
  \begin{equation*}
   \mathcal{I}_2\leq 2\nu\big(B_1^c\big) [\phi]_{\hdd} |x-y|^{\hdd}
   \leq 2\nu\big(B_1^c\big) [\phi]_{\hdd}|x-y|^{\hdd-2\sigma}.\qedhere
  \end{equation*}
 \end{proof}
 \subsection{Viscosity solutions}
 In this section we define viscosity solutions and give results for \autoref{eq:hjb}.
 Let $(t,x,\ell)\mapsto \mathcal{F}\big(t,x,\ell)$ and $w_0$ be continuous functions,
 and $\mathcal{F}$ be non-decreasing in $\ell$.
 For $\fL$ satisfying~\ref{L:deg} 
 consider the following problem
 \begin{align}\label{eq:viscosity}
  \left\{\begin{aligned}
   \dt w&= \mathcal{F}\big(t,x,(\fL w)(t,x)\big),\quad&\text{on $\T\times\X$},\\
   w(0) &= w_0,\quad&\text{on $\X$}.
  \end{aligned}\right.
 \end{align}
 For $0\leq r<\infty$ and $p\in\X$ we introduce linear operators
 \begin{equation*}
 \begin{split}
  \fLhigh \phi(x) 
  &= \int _{B_r^c}\big(\phi(x+z)-\phi(x)\big)\,\nu(dz),\\
  \fLlow \phi(x)  &= \int_{B_r} \big(\phi(x+z)-\phi(x)\big)\,\nu(dz),
  \end{split}
 \end{equation*}
 defined for
 bounded semicontinuous and $C^1$ functions respectively.
 
 \begin{definition}\label{def:viscosity}
  A bounded upper-semicontinuous function $u^-:\Tb\times\X\to\R$ is a \textit{viscosity subsolution}
  of \autoref{eq:viscosity} if 
  \begin{enumerate}
   \item $u^-(0,x)\leq w_0(x)$ for every $x\in\X$;
   \item\label{item:tf1} for every $r\in(0,1)$, test function $\phi\in C^1(\T\times\X)$,
   and a maximum point $(t,x)$ of $u^--\phi$ we have
   \begin{align*}\hspace{.5\mathindent}
    \dt \phi(t,x) - \mathcal{F}\Big(t,x,\big(c\cdot \nabla\phi 
    + \fLhigh (u^-)+\fLlow\phi\big)(t,x)\Big) \leq 0.
   \end{align*}
  \end{enumerate}
 \end{definition}
  A supersolution is defined similarly, replacing max, upper-semicontinuous, and ``$\leq$'' by min, lower-semicontinuous, and ``$\geq$''. A viscosity solution is a sub- and supersolution at the same time.
   \begin{remark}\label{rem:viscosity}
  \npar
  Bounded classical solutions
  are bounded viscosity solutions.
  \npar\label{rem:c1} 
  In \autoref{def:viscosity}\autoref{item:tf1},
  we may take 
    a test function $\phi\in \mathcal{X}$, where
  \begin{align*}
    \mathcal{X}  = \Big\{ \psi \in \Cb  \ : \ 
    \fL_1 \psi\in \Cb, \  \dt \psi\in\Cb  \Big\}.
  \end{align*}
  See \cite[\S10.1.2]{MR2597943} for first order PDEs and
  $\phi\in C^1(\T\times\X)$, and  the proof for \autoref{eq:viscosity} 
  and $\phi\in\mathcal{X}$ is a small modification of this.
  \npar In this paper we only consider functions $\mathcal{F}$ of the form
  \begin{align*}
   &(i) \quad \mathcal{F}(t,x,\ell) = F(\ell) - f(t,x); & &(ii)\quad
    \mathcal{F}(t,x,\ell) = b(t,x)\,\ell.
  \end{align*}
 \end{remark}
 \begin{definition}\label{def:comparison}
  The comparison principle holds for \autoref{eq:viscosity}
  if any subsolution $u^-$ and supersolution $u^+$
  satisfy $u^-(t,x)\leq u^+(t,x)$ for every  $(t,x)\in\Tb\times\X$.
 \end{definition}
  The next lemma shows that sufficiently regular viscosity solutions are classical.
 \begin{lemma}\label{lemma:classical}
  Assume~\ref{L:deg} and let $w$ be a viscosity solution of \autoref{eq:viscosity}.
  If the comparison principle holds for \autoref{eq:viscosity} and 
  \begin{align*}
  w\in\LC{2\sigma+\epsilon},\qquad\epsilon\in (0,1-2\sigma)
  \end{align*}
  then $\dt w\in\Cb$ and $w$ is a bounded classical solution of \autoref{eq:viscosity}.
 \end{lemma}
\begin{proof}
  \begin{step*}{We show $\fL w \in\Cb$} By \autoref{def:viscosity}, $w\in\Cbb$, and then since $w\in\LC{2\sigma+\epsilon}$, by \autoref{prop:lap-lip}
  $\fL w\in\LC{\epsilon}$.
  Let $(t_n,x_n)\to (t_0,x_0)$, and note that for every $(t,x)\in\T\times\X$,
   \begin{align*}
       \big|w(t,x+z)-w(t,x)\big|\leq 2\|w\|_\infty\mathbbm{1}_{B_1^c}(z) + \sup_{s\in\T}[w(s)]_{2\sigma+\epsilon}|z|^{2\sigma+\epsilon}\mathbbm{1}_{B_1}(z).
   \end{align*}
  The function on the right-hand side is $\nu$-integrable in $z$.
   Then by \ref{L:deg}, the Lebesgue dominated convergence theorem, and the continuity of $w$, we get
   \begin{multline*}
       \lim_{n\to\infty}\fL w(t_n,x_n) = \lim_{n\to\infty}\int_\X w(t_n,x_n+z)-w(t_n,x_n)\,\nu(dz)\\
       =\int_\X w(t_0,x_0+z)-w(t_0,x_0)\,\nu(dz) = \fL w(t_0,x_0).
   \end{multline*}
  \end{step*}

\begin{step*}{We show $w$ is a.e.~$t$-differentiable} Let $t_0\in\T$ be fixed  and define
  \begin{align*}
   u^{\pm}(t,x) = w(t_0,x)\pm L(t-t_0),\qquad L=\|\mathcal{F}\big(t,x,\fL w(t,x)\big)\|_\infty.
  \end{align*}
  Then $u^{+}$ and $u^-$ are respectively a viscosity supersolution and a subsolution of \autoref{eq:viscosity}  for $t\geq t_0$.
  Therefore, by the comparison principle,
  \begin{align*}
   u^-(t,x) \leq w(t,x) \leq u^+(t,x),\qquad\text{for every $(t,x)\in[t_0,T]\times\X$}.
  \end{align*}
   Hence $|w(t_0,x)-w(t,x)|\leq L|t-t_0|$, and $w$ is $t$‑Lipschitz.
  Thus, by  the theorems of Rademacher~\cite[\S5.8~Theorem~6]{MR2597943}
  and  Fubini~\cite[Theorem~7.6.5]{MR2267655},
  we find that $w$ is a.e.~$t$‑differentiable in $\T\times\X$.
  \end{step*}
  \begin{step*}{We show $\dt w\in \Cb$} Suppose $w$ is $t$‑differentiable at $(t_0,x_0)\in\T\times\X$.
  Consider
  \begin{align*}
     \phi(t,x) = |x-x_0|^{2\sigma +\epsilon}\big(1+\sup_{s\in\T}[w(s)]_{2\sigma+\epsilon}\big) +v(t),\quad \phi (t)\in C^{2\sigma+\epsilon}(B_1(x_0)),
  \end{align*}
  where $v\in C^1(\T)$ is the function we get from \cite[\S10.1.2, Lemma]{MR2597943} (only for $t$ variable at fixed $x_0$) such that
$v(t_0) = w(t_0,x_0)$ and $w(t,x_0)-v(t)$ attains a maximum at~$t_0$. 
  Then, $w-\phi$ has a strict local maximum at $(t_0,x_0)$ and  $\dt (w-\phi)(t_0,x_0)=0$.
  By the definition of a viscosity subsolution and \autoref{rem:viscosity}\autoref{rem:c1}, since $\phi \in \mathcal{X}$,
  \begin{equation}\label{eq:esti_subsolution}
   \dt w (t_0,x_0)= \dt \phi (t_0,x_0) \leq \mathcal{F}\big(t_0,x_0,\big(\fL_r \phi + \fL^r w\big) (t_0,x_0) \big),
  \end{equation}
  for every $r\in (0,1)$.
  By \ref{L:deg}, for every  $\psi \in C^{2\sigma+\epsilon}(B_1(x_0))$,
 \begin{align*}
 \big|\fL_r \psi (x_0)\big| & \leq [\psi]_{C^{2\sigma + \epsilon}(B_1(x_0))} \int_{B_r} |z|^{2\sigma + \epsilon} \nu(dz)   \leq \frac{K}{\epsilon} r^\epsilon  \, [\psi]_{C^{2\sigma + \epsilon}(B_1(x_0))} .
 \end{align*}
 Note that $(\fL^r w -\fL w)= \fL_r w$.
 Therefore, $\lim\limits_{r\to 0} \big(\fL_r \phi + \fL^r w\big) (t_0,x_0) = \fL w(t_0,x_0)$, as $\phi(t_0), w(t_0) \in C^{2\sigma+\epsilon}(B_1(x_0))$. By letting $r\to 0$ in \eqref{eq:esti_subsolution} we find
 \begin{align*}
   \dt w (t_0,x_0) \leq \mathcal{F}\big(t_0,x_0, \fL w (t_0,x_0) \big).
  \end{align*}
   A similar argument for supersolutions gives the opposite inequality.
  Then by a.e.~$t$‑differentiability we get
  \begin{align*}
  \dt w(t,x) = \mathcal{F}\big(t,x,\fL w(t,x)\big)\quad\text{a.e.~in $\T\times\X$.}
  \end{align*}
  Since $(t,x)\mapsto \mathcal{F}(t,x,\fL w (t,x))\in\Cb$, we obtain $\dt w\in\Cb$ and so $w$ is a bounded classical solution of \autoref{eq:viscosity}.\qedhere
  \end{step*}
 \end{proof} 
 
 \section{Proof of \autoref{thm:deg}---well-posedness}\label{section:hjb-fp}

We first recall the assumptions and the main result from \cite{CJK}, and then show that the assumptions of this theory hold for our degenerate problem.

\subsection{General well-posedness theory from \texorpdfstring{\cite{CJK}}{[CJK]}}
 Let 
  \begin{align*}
      \HJcD &= 
      \big\{\text{$u\in\Cbb$ is a bounded classical solution of \autoref{eq:hjb}}\\ &\qquad\text{with data $(f,\uT)=\big(\cF(m),\cG(m(T))\big)$}:  m\in\CaPX{\frac12}\big\}    
  \end{align*}
 and consider the following conditions:   \begin{description*}\vspace{0.33\baselineskip}
   \item[(S1)\label{R1}] For every $m\in\CaPX{\frac12}$ there exists a
   bounded classical solution $u$ of \autoref{eq:hjb} with data $(f,\uT)=\big(\cF(m),\cG(m(T))\big)$.
   \item[(S2)\label{R2}] If $\{u_n,u\}_{n\in\N}\subset \HJcD$
   and $\lim\limits_{n\to\infty}\|u_n- u\|_\infty = 0 $,
   then $\fL u_n(t) \to \fL u(t)$ uniformly on compact sets in $\X$ for every $t\in\T$.
   \item[(S3)\label{R3}] There is $\KHJ\geq0$ such that 
   $\|F'(\fL u)\|_{\infty}\leq \KHJ$ for every $u\in \HJcD$.
  \item[(S4)\label{U1}]
    It holds
   $\{\dt u,\,\fL u : u\in \HJcD\}\subset\Cb$.
  \item[(S5)\label{U2}] For each $b\in\{F'\big(\fL u\big): u\in \HJcD\}\cap\Cb$  
   and initial data $m_0\in\PX$ there exists at most one very weak solution of \autoref{eq:fp}.
 \end{description*}

 
 \begin{theorem}[{\cite[Theorem 2.9]{CJK}}]\label{thm:main}
  Assume~\ref{L:deg},~\ref{a:F1},~\ref{a:m}.
  If in addition
  \begin{enumerate}
      \item \ref{a:fg1}, \ref{R1}, \ref{R2}, \ref{R3} hold,
      then there exists a classical--very weak  solution of \autoref{eq:mfg};
      \item \ref{a:F2},~\ref{a:fg2}, \ref{U1}, \ref{U2} hold,
      then \autoref{eq:mfg} has at most one classical--very weak solution.
  \end{enumerate}
\end{theorem}
 In the next section we study the HJB equation \eqref{eq:hjb} and prove that conditions \ref{R1}--\ref{U1} hold under the assumptions of \autoref{thm:deg}. Condition~\ref{R1} describes existence of solutions of the HJB equation \eqref{eq:hjb} 
 (which are unique by 
 \autoref{thm:hjb-viscosity} below).
 Conditions~\ref{R2}, \ref{R3}, \ref{U1} describe various (related) properties of solutions of \autoref{eq:hjb}.
 
 Then in \autoref{pf:fp} we show that condition \ref{U2}---uniqueness of solutions---holds for the FP equation \eqref{eq:fp}. Existence of solutions is proved in \cite[Theorem 6.6]{CJK}
 
 \subsection{Estimates on the HJB equation \texorpdfstring{\eqref{eq:hjb}}{(3)}}
 We recall the following uniqueness, stability, and existence result for viscosity solutions of \autoref{eq:hjb}.
 \begin{theorem}
 [{\cite[Theorem 5.3]{CJK}}]
 \label{thm:hjb-viscosity}
  Assume~\ref{L:deg}
  ,~\ref{a:F1}, and $(f,\uT)$
  are bounded and continuous.
   \begin{enumerate}
\item The comparison principle (see \autoref{def:comparison}) holds for \autoref{eq:hjb}.
\item\label{item:cp} Let $u_1,u_2$ be viscosity solutions of \autoref{eq:hjb}
  with bounded uniformly continuous data $(f_1,\unT{1})$, $(f_2,\unT{2})$, respectively.
  Then for every $t\in\Tb$,
    \begin{align*}\hspace{.5\mathindent}
     \|u_1(t)-u_2(t)\|_{\infty} \leq (T-t) \|f_1-f_2\|_{\infty} + \|\unT{1} - \unT{2}\|_{\infty}.
    \end{align*} 
\item\label{item:vs-ex} There exists a unique viscosity solution of \autoref{eq:hjb}.   
\end{enumerate}
  \end{theorem}
 By restricting the range to which $(f,g)$ belongs,
 we can then derive the desired properties of solutions.
   \begin{theorem}
   \label{thm:hjb-lipschitz}   
    Assume \ref{L:deg}, \ref{a:F1} and $(f,\uT)\in \cD_0(\hd,M)$ (as in \ref{D:deg}).
    If $u$ is a viscosity solution of \autoref{eq:hjb}, then
    \begin{enumerate}
    \item\label{item:lipschitz}
    $u\in \LC{\hd}$ with 
    $\max\{\|u(t)\|_\infty,\,[u(t)]_\hd\}\leq M(T-t+1)$,
    and \mbox{$\fL u\in\LC{\hd-2\sigma}$} with
    \begin{align*}\hspace{.5\mathindent} 
    \|\fL u(t)\|_{\hd-2\sigma}\leq 4\Big(\frac{K}{\hd-2\sigma}+\nu\big(B_1^c\big)\Big)M(T-t+1);
    \end{align*}
   \item\label{item:unique} 
    $u$ is a bounded classical solution of \autoref{eq:hjb} and 
    \begin{align*}\hspace{.5\mathindent} 
    \dt u, \fL u \in\Cb;
    \end{align*}
   \item\label{item:limit} if $u_n$ are viscosity solutions of \autoref{eq:hjb},
     with data $(f_n, \unT{n})\in \cD_0(\hd,M)$
     and we have \mbox{$\lim\limits_{n\to\infty}\|u_n-u\|_{\infty}=0$},
     then $\fL u_n(t) \to\fL u(t)$ uniformly on compact sets in $\X$ for every $t\in\T$.
  \end{enumerate}
 \end{theorem}
 \begin{proof*}
 \begin{part*}\label{part:hjb2}
  Let $y \in\X $ and define $\widetilde u(t,x) = u(t, x +y)$.
  Notice that $\widetilde u$ is a viscosity solution of \autoref{eq:hjb} with data
  $(\widetilde f,\widetilde \uT)$, where $\widetilde f(x) = f(x+y)$ and $\widetilde \uT(x) = \uT(x+y)$.
  Hence by \autoref{thm:hjb-viscosity}\autoref{item:cp} for every~$(t,x,y)\in\Tb\times\X\times\X$ we get
  \begin{align*}
   |u(t,x+y)-u(t,x)| & \leq  (T-t)\| f - \widetilde f\|_\infty + \| \uT - \widetilde \uT\|_\infty \\
   & \leq |y|^\hd\big((T-t)[f]_\hd + [\uT]_\hd\big) \leq M(T-t+1)|y|^\hd.
  \end{align*}
   Since 
   $\|u(t)\|_\infty \leq (T-t)\|f\|_\infty + \|\uT\|_\infty\leq M(T-t+1)$ by \autoref{thm:hjb-viscosity}\autoref{item:cp} again, it follows that
  $u\in \LC{\hd}$.
  Then by \ref{L:deg} and \autoref{prop:lap-lip} we find that 
  \begin{align*}
   \|\fL u(t)\|_{\hd-2\sigma} \leq 4\Big(\frac{K}{\hd-2\sigma}+\nu\big(B_1^c\big)\Big)M(T-t+1).
  \end{align*}
  \vspace{-\baselineskip}
  \end{part*}
  \begin{part*}
   It follows from \autoref{part:hjb2} and \autoref{lemma:classical}
   that $u$ is a bounded classical solution and $\dt u,\fL u\in\Cb$.
 \end{part*}
  \begin{part*}
   By \autoref{part:hjb2} and the Arzel\`a--Ascoli  theorem,
  for every $t\in\T$ there exist a subsequence $\{u_{n_k}\}$ and a function $v\in C_b(\X)$, such that
  $\fL u_{n_k}(t) \to v$ uniformly on compact sets in $\X$.
  On the other hand, by \autoref{part:hjb2} and the Lebesgue dominated convergence theorem, $\lim\limits_{n\to \infty} \fL u_n(t,x) = \fL u (t,x)$
  for every $(t,x)\in\T\times\X$.
  Hence we find $\fL u_{n_k}(t) \to \fL u(t)$ uniformly on compact sets in $\X$ for every $t\in\T$.\qedhere 
  \end{part*}
 \end{proof*}
 
 \begin{corollary}\label{cor:hjb-l1-s}
  Assume~\ref{L:deg},~\ref{D:deg}, and~\ref{a:F1}.
   Then  conditions \ref{R1}, \ref{R2}, \ref{R3}, and \ref{U1} are satisfied.
 \end{corollary}
 \begin{proof}
  Conditions \ref{R1} and \ref{U1} are consequences of \autoref{thm:hjb-viscosity}\autoref{item:vs-ex} and
  \autoref{thm:hjb-lipschitz}\autoref{item:unique}, while \ref{R3} follows from
  \autoref{thm:hjb-lipschitz}\autoref{item:lipschitz} and \ref{a:F1}.
  We obtain \ref{R2} from \autoref{thm:hjb-lipschitz}\autoref{item:limit}.
 \end{proof}
 
 \subsection{Uniqueness for the FP equation \texorpdfstring{\eqref{eq:fp}}{(4)}}\label{pf:fp} 
Uniqueness for \autoref{eq:fp} follows from a Holmgren argument using as a test function the solution of the dual equation\vspace{-\baselineskip}
\begin{equation}\label{eq:fp-dual}
 \left\{
 \begin{aligned}
   \dt w &- b\fL w=0\quad&\text{on $\T\times\X$},\\
   w(0) &= \phi\quad&\text{on $\X$}.
 \end{aligned}
 \right.
\end{equation} 
Most of the proof consists of showing existence of solutions of this dual equation.
\begin{theorem}\label{thm:ext-holder-sol-deg}
  Assume \mbox{$2\sigma\in\big(0,\frac{1}{2}\big)$},
  \mbox{$\hb\in\big(\frac{2\sigma}{1-2\sigma},1\big]$},~\ref{L:deg},~\ref{a1'':b}, and
  \mbox{$\phi\in C^1_b(\X)$}.
  Then there exists a viscosity solution $w$ of \autoref{eq:fp-dual} such that 
  $w\in\LC{\hb_0}$ for {$\hb_0 = \hb-\frac{2\sigma}{1-2\sigma}$}.
  If in addition $\nu$ is symmetric at the origin, then 
  we can take \mbox{$\hb\in\big(\frac{2\sigma}{1-\sigma},1\big]$} and {$\hb_0 = \hb-\frac{2\sigma}{1-\sigma}$}.
 \end{theorem}
 The proof is given in \autoref{section:dual}. Under stronger assumptions on $2\sigma$ and $\beta$, this solution is classical.
 \begin{lemma}\label{thm:ext_reg_soln_fp-dual}
  Assume~\ref{L:deg} and  \ref{a1'':b}. If either
  \begin{enumerate}
      \item $2\sigma\in\big(0,\frac{3-\sqrt{5}}{2}\big)$ and $\hb\in \big(2\sigma+\tfrac{2\sigma}{1-2\sigma},1\big]$; or
      \item $\nu$ is symmetric at the origin,
       $2\sigma\in\big(0,\frac{5-\sqrt{17}}{2}\big)$,
       and $\hb\in \big(2\sigma+\tfrac{2\sigma}{1-\sigma},1\big]$;
  \end{enumerate}
  then there exists a bounded classical solution of \autoref{eq:fp-dual}.
 \end{lemma}
 \begin{proof}
 \begin{part*}\label{part:deg}
  By \autoref{thm:ext-holder-sol-deg}, there is a viscosity solution $w$ of \autoref{eq:fp-dual} such that \mbox{$w\in\LC{\hb_0}$},
  where $\hb_0=\hb-\frac{2\sigma}{1-2\sigma}>2\sigma$.
  By \autoref{prop:lap-lip} and \autoref{lemma:classical} it follows that $\fL w\in\LC{\hb_0-2\sigma}$ and 
  $w$ is a bounded classical solution of \autoref{eq:fp-dual}.
 \end{part*}
 \begin{part*}
 It follows by a similar argument as in \autoref{part:deg}, using the symmetric case of \autoref{thm:ext-holder-sol-deg}.\qedhere
 \end{part*}
 \end{proof}

In view of the critical result of \autoref{thm:ext_reg_soln_fp-dual}, the rest of the uniqueness proof is the short Holmgren argument given in the proof of \cite[Theorem 6.7]{CJK}. We reproduce it here for the sake of readability and completeness.
 
 \begin{proof}[Proof of \autoref{thm:fp-uniqueness}]
  Let $\varphi \in C_c^\infty(\X)$ and $t_0\in(0,T]$, and
  take $\widetilde{b}(t) = b(t_0-t)$ for every $t\in[0,t_0]$.
  Replace $b$ by $\widetilde{b}$ in \autoref{eq:fp-dual}.
  Then by \autoref{thm:ext_reg_soln_fp-dual} there exists a bounded classical solution $\widetilde{w}$ of \autoref{eq:fp-dual}.
Let $w(t)= \widetilde{w}(t_0-t)$ for $t\in [0,t_0]$.
  Then $\dt w, \fL w\in C\big((0,t_0)\times\X\big)$, and $w$ is a bounded classical solution of 
  \begin{align}\label{eq:fp-dual-back-nondeg}
   \left\{
   \begin{aligned}
    \dt w(t) &+ b(t)\fL w(t)=0\quad\text{in $(0,t_0)\times \X$},\\
    w(t_0) &= \varphi.
   \end{aligned}
   \right.
  \end{align}

  Suppose $m$ and $\widehat{m}$ are two very weak solutions of \autoref{eq:fp}
  with the same initial condition $m_0$ and coefficient $b$.
  Since $m$ is a distributional solution of \autoref{eq:fp}
  and $w$ satisfies~\eqref{eq:fp-dual-back-nondeg},
  \begin{align*}
   \big(m(t_0)-\widehat{m}(t_0)\big)[\varphi]-0
   =\int_0^{t_0} \big(m(\tau)-\widehat{m}(\tau)\big)\big[\dt w + b \fL(w)\big]\,d\tau=0.
  \end{align*}
  Since $t_0$ and $\varphi$ were arbitrary, this means $m(t)=\widehat m(t)$ in $\PX$ for $t\in(0,T)$.
 \end{proof}

 \begin{corollary}\label{cor:b-uniq}
  Assume~\ref{a:F1}, \ref{a:m}, \ref{D:deg},~\ref{L:deg}.
  Then condition \ref{U2} is satisfied if either
  \begin{enumerate}
      \item
      $\frac{2\sigma}{\hd-2\sigma}\big(1+\frac1{1-2\sigma}\big)<\gamma$; or
      \item
      $\nu$ is symmetric at the origin\footnotemark[\value{footnote}], and
      $\frac{2\sigma}{\hd-2\sigma}\big(1+\frac1{1-\sigma}\big)<\gamma$;
  \end{enumerate}
 \end{corollary}
 \begin{proof*}
  Let $u_1,\,u_2$ be bounded classical solutions of \autoref{eq:hjb} and $v_1=\fL u_1$, $v_2=\fL u_2$.
  Since $F'\in \Holder{\gamma}(\R)$ by \ref{a:F1}, we may consider 
  \begin{align*}
   b(t,x) = \int_0^1 F'\big(sv_1(t,x)+(1-s)v_2(t,x)\big)\,ds.
  \end{align*}
  Because $u_1,\,u_2$ are bounded classical solutions and $F'\geq0$, we have $b\in\Cu$ and
  $b\geq0$.
  \smallskip\par
  \begin{part*}\label{part:b-uniq1}
  By \autoref{thm:hjb-lipschitz}\autoref{item:lipschitz},  $v_1,\,v_2 \in\LC{\hd-2\sigma}$.
  Thus $b$ satisfies~\ref{a1'':b} with \mbox{$\hb = \gamma(\hd-2\sigma)$}.
  Then, since $\frac{2\sigma}{(\hd-2\sigma)}\big(1+\frac1{1-2\sigma}\big)<\gamma$,
  we have $\hb >2\sigma+ \frac{2\sigma}{1-2\sigma}$
  and \ref{U2} follows from \autoref{thm:fp-uniqueness}\autoref{item:fp-uniq-deg}.
  \end{part*}
  \begin{part*}
  We proceed as in \autoref{part:b-uniq1} and use \autoref{thm:fp-uniqueness}\autoref{item:fp-uniq-deg1}.\qedhere
  \end{part*}
  \end{proof*}

\section{Proof of \autoref{thm:ext-holder-sol-deg}---Dual of degenerate FP equation}\label{section:dual}
We employ viscosity solutions theory and a non-standard doubling of variables argument to prove existence
 of H\"older-continuous solutions of \autoref{eq:fp-dual}. 
 
 \begin{lemma} \label{lem:fp-dual-comparison}
  Assume \mbox{$2\sigma\in\big(0,\frac{1}{2}\big)$},
  \mbox{$\hb\in\big[\frac{2\sigma}{1-2\sigma},1\big]$},~\ref{L:deg},~\ref{a1'':b}, and \mbox{$\phi\in C_b(\X)$}.
  \begin{enumerate}
\item \label{thm:comparison_principle} The comparison principle holds for \autoref{eq:fp-dual} (see \autoref{def:comparison}).
      \item \label{cor:fp-dual-existence}
  There exists a viscosity solution $w$ of \autoref{eq:fp-dual}.  \end{enumerate}
  \end{lemma}
 \noindent \textit{Proof of \hyperref[thm:comparison_principle]{Part {\rm(}i{\rm)}}.}
  \begin{step}\label{step:max}
For every $\epsilon,\delta>0$ let
 \begin{align}\label{eq:psi-epsilon}
  \psi_{\epsilon,\delta}(x,y) = \frac{|x-y|^2}{\epsilon}+\delta\big(V(x)+V(y)\big),
 \end{align}
 where $V$ is a Lyapunov function such that $\|\fL V\|_\infty<\infty$  (see \autoref{lemma:levy-lyapunov}). Let  $\{a_{\epsilon,\delta}\}_{\epsilon,\delta}$ be a bounded set\footnote{In \autoref{step:contra} we will simply take  $a_{\epsilon,\delta}=M$; we write a slightly more general version than needed here to reuse some of the results in the proof of \autoref{thm:ext-holder-sol-deg} below.} and $u$, $v$ be a viscosity sub- and super-solution, respectively. Define
 \begin{align}\label{eq:Psi-epsilon}
  \Psi_{\eta,\epsilon,\delta}(t,x,s,y)
  =  u(t,x) -v(s,y)-\psi_{\epsilon,\delta}(x,y) - \frac{|t-s|^2}{\eta}
  - \frac{a_{\epsilon,\delta}}{4T}(t+s).
 \end{align}
 Following the usual argument in the doubling of variables technique (cf.~\cite[Theorem 4.1]{MR2129093}) we can show that  for every $\eta,\epsilon,\delta>0$ the function $\Psi_{\eta,\epsilon,\delta}$ has 
  a maximum point $(t_*,x_*,s_*,y_*)\in\big(\Tb\times\X\big)^2$ such that
  \begin{align}\label{eq:square-estimate}
    \frac{|x_*-y_*|^2}{\epsilon} + \frac{|t_*-s_*|^2}{\eta} \leq  \frac{u(t_*,x_*) -v(s_*,y_*) -u(s_*,y_*) + v(t_*,x_*)}{2},
  \end{align}
  and for every $\delta>0$ there exist subsequence $\eta_k$ such that for every $\epsilon>0$
  \begin{align}\label{eq:t-limit}
   \lim_{\eta_k\to0} \frac{|t_* - s_*|^2}{\eta_k}
   = 0\quad\text{and}\quad \lim_{\eta_k\to0}(t_*,x_*,s_*,y_*) 
   = (t_{\epsilon,\delta},x_{\epsilon,\delta},t_{\epsilon,\delta},y_{\epsilon,\delta}),
  \end{align}
  and a subsequence $\epsilon_n$ such that
  \begin{align}\label{eq:x-limit}
   \lim_{\epsilon_n\to0}\lim_{\eta_k\to0} \frac{|x_* - y_*|^2}{\epsilon_n}= 0
   \quad\text{and}\quad
   \lim_{\epsilon_n\to0}\lim_{\eta_k\to0}(t_*,x_*,s_*,y_*)
   = (t_{\delta},x_{\delta},t_{\delta},x_{\delta}),
  \end{align}
  where $(t_{\epsilon,\delta},x_{\epsilon,\delta},t_{\epsilon,\delta},y_{\epsilon,\delta})$
  and $(t_{\delta},x_{\delta},t_{\delta},x_{\delta})$
  denote the respective limit points.
\end{step} 
\begin{step}\label{step:holder}
  Let $t_*,s_*>0$. The maximum point of $\Psi_{\eta,\epsilon,\delta}$ is  $(t_*,x_*,s_*,y_*)$. By taking
  \begin{align*}
      v(s_*,y_*) + \psi_{\epsilon,\delta}(x,y_*)+\frac{|t-s_*|^2}{\eta} +\frac{a_{\epsilon,\delta}(t+s_*)}{4T},
      \end{align*} and
      \begin{align*}
         u(t_*,x_*) - \psi_{\epsilon,\delta}(x_*,y) -\frac{|t_*-s|^2}{\eta} -\frac{a_{\epsilon,\delta}(t_*+s)}{4T},
      \end{align*}
 as test functions in \autoref{def:viscosity} of viscosity sub- and supersolutions, we have
  \begin{align*}
  \begin{split}
   \frac{a_{\epsilon,\delta}}{4T} + \frac{2(t_*-s_*)}{\eta} 
    - b(t_*,x_*)\Big(\fLhigh u(t_*,x_*)+\fLlow\psi_{\epsilon,\delta}(\,\cdot\,,y_*)(x_*)\Big)\leq 0,\\
   -\frac{a_{\epsilon,\delta}}{4T} + \frac{2(t_*-s_*)}{\eta} 
    - b(s_*,y_*)\Big(\fLhigh v(s_*,y_*)-\fLlow\psi_{\epsilon,\delta}(x_*,\,\cdot\,)(y_*)\Big)\geq 0.
  \end{split}
  \end{align*}
   We subtract these two inequalities and obtain 
  \begin{align}\label{eq:a-estimate}
  \begin{split}
   \frac{a_{\epsilon,\delta}}{2T}
   &\leq b(s_*,y_*)\fLlow\psi_{\epsilon,\delta}(x_*,\,\cdot\,)(y_*)
     +b(t_*,x_*)\fLlow\psi_{\epsilon,\delta}(\,\cdot\,,y_*)(x_*) \\
   &\qquad+ b(t_*,x_*)\Big(\big(\fLhigh u\big)(t_*,x_*)-\big( \fLhigh v\big)(s_*,y_*)\Big)\\
   &\qquad+ \big(b(t_*,x_*)-b(s_*,y_*)\big)\big(\fLhigh v\big)(s_*,y_*).
   \end{split}
  \end{align}
   Observe that since
  $\Psi_{\eta,\epsilon,\delta}(t_*,x_*+z,s_*,y_*+z)\leq \Psi_{\eta,\epsilon,\delta}(t_*,x_*,s_*,y_*)$, 
  we have
  \begin{multline*}
   u(t_*,x_*+z) -v(s_*,y_*+z) - u(t_*,x_*) + v(s_*,y_*)
   \\ \leq {\delta}\, \big(V(x_* + z) + V(y_* + z)-V(x_*)-V(y_*)\big).
  \end{multline*}
  Therefore, because of~\ref{L:deg}, for every $r\in(0,1)$
  \begin{align}\label{eq:estimate-high1}
    \big(\fLhigh u\big)(t_*,x_*)-\big(\fLhigh v\big)(s_*,y_*) 
    \leq 2\delta \bigg(\|\fL V\|_\infty+\|\nabla V\|_\infty\int_{B_r}|z|\,\nu(dz)\bigg).
  \end{align}
  We also find that
  \begin{align}\label{eq:estimate-high2}
    \big(\fLhigh v\big)(t_*,y_*)
    \leq 2\|v\|_\infty \bigg(\nu\big(B_1^c\big)
    +\int_{B_1\setminus B_r}\,\nu(dz)\bigg).
  \end{align}
  Again by~\ref{L:deg} and the Cauchy--Schwarz inequality,
  \begin{align}\label{eq:L_r}
  \begin{split}
  \big(\fLlow \psi_{\epsilon,\delta}(x_*,\cdot\,)\big)(y_*)
  &= \int _{|z|\leq r} \frac{|x_*-y_*+z|^2-|x_*-y_*|^2}{\epsilon}\,\nu(dz) + {\delta} \fLlow V(y_*)  \\
&\leq  \int _{|z|\leq r}\bigg(\frac{\big||z|^2+2(x_*-y_*)\cdot z\big|}{\epsilon}
        +{\delta\,\|\nabla V\|_\infty|z|}\bigg)\,\nu(dz)\\
  &\leq\frac{K}{1-2\sigma}\Big(\delta+\frac{2|x_*-y_*|+r}{\epsilon}\Big)r^{1-2\sigma}.
  \end{split}
  \end{align}
 
  By using~\eqref{eq:estimate-high1},~\eqref{eq:estimate-high2}, \eqref{eq:L_r}, and \ref{L:deg}
  in inequality~\eqref{eq:a-estimate}, we obtain
  \begin{align*}
   \frac{a_{\epsilon,\delta}}{2T}&\leq
   \frac{2 K\|b\|_\infty}{1-2\sigma}\bigg(2 \delta+\frac{2|x_*-y_*|+r}{\epsilon}\bigg)r^{1-2\sigma}
   +2\delta\|b\|_\infty\|\fL V\|_\infty\\
   &\qquad+ 2\|v\|_\infty \big|b(t_*,x_*)-b(s_*,y_*)\big|\bigg(\nu\big(B_1^c\big)
   +\frac{K}{1-2\sigma}r^{-2\sigma}\bigg).
 \end{align*}
 Taking the limit as $\eta_k \to 0$, using the second part of \eqref{eq:t-limit}, and then~\ref{a1'':b}, we get
 \begin{align} \label{eq:doubling-est}
   \frac{a_{\epsilon,\delta}}{2T}
   \leq C\bigg(\frac{|x_{\epsilon,\delta}-y_{\epsilon,\delta}|+r}{\epsilon}r^{1-2\sigma}
   + |x_{\epsilon,\delta}-y_{\epsilon,\delta}|^{\hb} \big(1 + r^{-2\sigma}\big)+  \delta\bigg),
  \end{align}
 where constant $C>0$ is independent of $r$.
\end{step}
\begin{step}\label{step:contra}
   Denote
  \begin{align*}
   M_0 = \sup_{x\in\X}\Big(u(0,x)-v(0,x)\Big)\quad\text{and}\quad
   M= \sup_{t\in\Tb\vphantom{x\in\X}}\sup_{x\in\X\vphantom{t\in\Tb}} \Big(u(t,x)-v(t,x)\Big),
  \end{align*}
  and assume  by contradiction that $M_0\leq 0$ and  $M>0$.
  Because the functions $u$ and $v$ are bounded,
  we also have $M<\infty$.

  Take $a_{\epsilon,\delta} =M$. Suppose $\lim\limits_{\epsilon_n,\eta_k\to0}t_* = t_\delta = 0$.
  Then for every $\delta>0$ we have
  \begin{align}\label{eq:m-contradiction}
   \limsup_{\epsilon_n\to0}\limsup_{\eta_k\to 0}\Psi_{\eta_k,\epsilon_n,\delta}(t_*,x_*,s_*,y_*)
   \leq  \Phi(0,x_{\delta},0,x_{\delta}) \leq M_0 \leq 0.
  \end{align} 
  On the other hand, by the definition of $M$, there exists a point $(t_M,x_M)$ such that
  $\Phi(t_M,x_M,t_M,x_M)\geq \frac{3}{4} M$.
  Take $\delta>0$ such that $\delta V(x_M) \leq\tfrac{1}{16}M$.
  Then we get
  \begin{multline*}
   \Psi_{\eta_k,\epsilon_n,\delta}(t_*,x_*,s_*,y_*)
   \geq \Phi(t_M,x_M,t_M,x_M)-2\delta V(x_M) - \tfrac{1}{2}M\\
   \geq M \big(\tfrac34-\tfrac18-\tfrac12\big) = \tfrac18 M>0.
  \end{multline*}
  This contradicts~\eqref{eq:m-contradiction} and shows that $t_\delta>0$
  for $\delta \leq\frac{M}{16 V(x_M)}$.
  Hence, without loss of generality, we may assume $t_*,s_*>0$. We then put $r^{2\sigma} =\epsilon_n^{\hb/2}$ in \eqref{eq:doubling-est} (see \autoref{rem:optimal1}\autoref{rem:optimal-nonsym1}) and get
  \begin{multline*}
   \frac{M}{2CT}\leq \epsilon_n^{\frac{\hb(1-\sigma)-2\sigma}{2\sigma}}
   +\bigg(\frac{|x_{\epsilon_n,\delta}-y_{\epsilon_n,\delta}|^2}{\epsilon_n}\bigg)^{\frac12}
   \epsilon_n^{\frac{\hb(1-2\sigma)}{4\sigma}-\frac12}\\
   +\bigg(\frac{|x_{\epsilon_n,\delta}-y_{\epsilon_n,\delta}|^2}{\epsilon_n}\bigg)^{\frac{\hb}{2}}
   +|x_{\epsilon_n,\delta}-y_{\epsilon_n,\delta}|^\hb+\delta.
  \end{multline*}
  Assumption $\hb\geq\frac{2\sigma}{1-2\sigma}$ is equivalent to
  $\frac{\hb(1-2\sigma)}{4\sigma}-\frac12\geq 0$ (see \autoref{rem:optimal1}\autoref{rem:optimal-symmetric1}).
  By \eqref{eq:x-limit} we have 
  ${\epsilon_n^{-1}}{|x_{\epsilon_n,\delta}-y_{\epsilon_n,\delta}|^2}\to 0$,
  thus the expression on the right-hand side converges to $\delta$ as $\epsilon_n\to0$.
  Since $\delta$ is arbitrary, we obtain $M \leq 0$, which is a contradiction. That completes the proof of \hyperref[thm:comparison_principle]{Part (\emph{i})}. \qed
\end{step} 

\noindent\textit{Proof of \hyperref[cor:fp-dual-existence]{Part {\rm(}ii{\rm)}}.}
 Notice that $u\equiv -\|\phi\|_\infty$ and $v\equiv \|\phi\|_\infty$ are a subsolution and a supersolution of \autoref{eq:fp-dual}, respectively.
  Using \autoref{lem:fp-dual-comparison}\autoref{thm:comparison_principle},
  existence of a (unique) bounded continuous viscosity solution follows by the Perron method
  (cf.~e.g.~the proof of~\cite[Theorem~2.3]{MR2653895} for a similar result). \qed

\begin{remark}\label{rem:optimal1}
   Eventually we want to put $b=F'(\fL u)$.
   In the best case, both $u$ and $F'$ may be Lipschitz and then $\hb=1-2\sigma$.
  Our aim is to obtain the most lenient estimate on $\sigma$ in terms of~$\hb$.
  \npar\label{rem:optimal-nonsym1}
  To this end, we cannot do better than substituting $r^{2\sigma} =\epsilon^{\hb/2}$.
  If $a$ is such that $r =\epsilon^{a}$, then we need to have $\frac\hb2 -2\sigma a\geq0$ and 
  at the same time $(1-2\sigma)a-\frac12\geq0$.
  Combining both inequalities we obtain $\frac\hb{4\sigma}\geq a\geq \frac{1}{2(1-2\sigma)}$
  and we still get $\hb\geq\frac{2\sigma}{1-2\sigma}$.
  When $\hb=1-2\sigma$, this translates to $\sigma\in\big(0,\frac{3-\sqrt{5}}{4}\big]$,
  and $\frac{3-\sqrt{5}}{4}\approx \frac{1}{5}$.
 \npar\label{rem:optimal-symmetric1}
  When $\nu$ is symmetric
  at the origin, we may obtain a better estimate on $\sigma$. 
  In this case $\int _{|z|\leq r} \big((x_*-y_*)\cdot z\big)\,\nu(dz) =0$, 
  which leads to an improved version of \eqref{eq:L_r} and~\eqref{eq:doubling-est}:
  \begin{align*}
   \frac{M}{2CT}\leq 
   \epsilon_n^{-1}r^{2-2\sigma}
   + |x_{{\epsilon_n},\delta}-y_{\epsilon_n,\delta}|^\hb \big(1 + r^{-2\sigma}\big)+\delta.
  \end{align*}
  Under the same scaling $r^{2\sigma} =\epsilon^{\hb/2}$,
  the dominant exponent is then ${\frac{\hb(1-\sigma)-2\sigma}{2\sigma}}$.
  It has to be strictly positive, hence $\hb>\frac{2\sigma}{1-\sigma}$.
  When $\hb=1-2\sigma$, this translates to
  $\sigma\in\big(0,\frac{5-\sqrt{17}}{4}\big)$, and $\frac{5-\sqrt{17}}{4}\approx\frac{2}{9}$.
 \end{remark} 
 \begin{proof}[Proof of \autoref{thm:ext-holder-sol-deg}]
  \begin{step}\label{step:alpha}
  For $\epsilon,\delta>0$, let $\psi_{\epsilon,\delta}$ be given by \eqref{eq:psi-epsilon} and define 
  \begin{align*}
   M^0_{\epsilon,\delta} 
   = \sup_{(x,y)\in\X\times\X} \Big\{ w(0,x)- w(0,y) -\psi_{\epsilon,\delta}(x,y) \Big\}
   \end{align*}
  and
  \begin{align*}
   M_{\epsilon,\delta} 
   = \sup_{t\in\Tb\vphantom{\X}}\sup_{(x,y)\in\X\times\X\vphantom{\Tb}}
     \Big\{ w(t,x)-w(t,y) -\psi_{\epsilon,\delta}(x,y) \Big\},
  \end{align*}
  where $w$ is the viscosity solution of \autoref{eq:fp-dual}. Note $0\leq \big(M_{\epsilon,\delta} - M^0_{\epsilon,\delta}\big)\leq 4\|w\|_\infty$,
  and for~$(t,x,y)\in\Tb\times\X\times\X$ it holds that
  \begin{align}\label{eq:ineq_M_epsilon}
   w(t,x)-w(t,y) \leq M_{\epsilon,\delta} + \psi_{\epsilon,\delta}(x,y).
  \end{align}

   We have $w(0)=\phi\in\Hb{1}$ and without loss of generality
   (since the equation is linear) we may assume $[\phi]_1\leq 1$.
   Then
  \begin{align*}
   w(0,x) -w(0,y) - \psi_{\epsilon,\delta}(x,y)  
   \leq |x-y| - \frac{|x-y|^2}{\epsilon}\leq \frac{\epsilon}{4},
  \end{align*}
  and thus $M^0_{\epsilon,\delta} \leq  \epsilon/4$.
  Take ${\Psi}_{\eta,\epsilon,\delta}$ as in~\eqref{eq:Psi-epsilon} with $u=v=w$  and $a_{\epsilon,\delta} = M_{\epsilon,\delta}- M^0_{\epsilon,\delta}$. 
  The results of Steps \ref{step:max} and \ref{step:holder} in the proof of \autoref{lem:fp-dual-comparison}\autoref{thm:comparison_principle} remain valid.
 
  Let us fix $\epsilon,\delta>0$.
  If $t_{\epsilon,\delta}=0$, then by the definition of $M_{\epsilon,\delta}$, $\Psi_{\eta,\epsilon,\delta}$ and $M^0_{\epsilon,\delta}$
  \begin{align*}
    & \frac{M_{\epsilon,\delta} + M^0_{\epsilon,\delta}}{2}= M_{\epsilon,\delta} - \frac{M_{\epsilon,\delta} - M^0_{\epsilon,\delta}}{2} 
    \leq \lim_{\eta_k\to 0}\Psi_{\eta,\epsilon,\delta}(t_*,x_*,s_*,y_*) \\
    & 
   =w(0,x_{\epsilon,\delta})-w(0,y_{\epsilon,\delta})
    -\psi_{\epsilon,\delta}(x_{\epsilon,\delta},y_{\epsilon,\delta})
   \leq M^0_{ \epsilon,\delta}.
  \end{align*} 
  Again, as $M^0_{\epsilon,\delta} \leq M_{\epsilon,\delta}$ by definition,  we get $M_{\epsilon,\delta} = M^0_{\epsilon,\delta}$.
  Because of~\eqref{eq:ineq_M_epsilon},
  for every~$(t,x,y)\in\Tb\times\X\times\X$ we thus have
  \begin{align}\label{eq:mm-are-equal}
   w(t,x)-w(t,y) 
   \leq M^0_{\epsilon,\delta} + \psi_{\epsilon,\delta}(x,y)
   \leq \frac{\epsilon}{4}+\frac{|x-y|^2}{\epsilon}+\delta\big(V(x)+V(y)\big).
  \end{align}
  
  In turn, if $t_{\epsilon,\delta}>0$, then without loss of generality we may assume that $t_*,s_*>0$.
 By \eqref{eq:doubling-est}  we therefore obtain
  \begin{align}\label{eq:mm-estimate}
   \frac{M_{\epsilon,\delta}-M^0_{\epsilon,\delta}}{2CT} 
   \leq \frac{r+|x_{\epsilon,\delta}-y_{\epsilon,\delta}|}{\epsilon}r^{1-2\sigma}
   + |x_{\epsilon,\delta}-y_{\epsilon,\delta}|^{\hb} \big(1 + r^{-2\sigma}\big) + \delta.
  \end{align}
  We also have ${|x_{\epsilon,\delta}-y_{\epsilon,\delta}|^2}\leq 2\epsilon\|w\|_\infty$,
  thanks to~\eqref{eq:square-estimate}.
  Thus, by combining~\eqref{eq:ineq_M_epsilon}, ~\eqref{eq:mm-estimate} and that $M^0_{\epsilon,\delta}\leq \epsilon/4$, we get for all $(t,x,y)\in\Tb\times\X\times\X$
  \begin{align*}
   w(t,x)-  & w(t,y)
  \leq  \frac{\epsilon}{4} + \frac{|x-y|^2}{\epsilon} 
   +  \delta\, \Big(V(x) + V(y) + 2CT\Big) \\
   & +2CTr^{-2\sigma}\Big(\frac{r^2}{\epsilon} + \sqrt{\frac{2\|w\|_\infty r^2}{\epsilon}}
   + \|2w\|_\infty^{\hb/2}\epsilon^{\hb/2}(1+r^{2\sigma})\Big).
  \end{align*} 
  Since the right-hand side dominates the right-hand side in~\eqref{eq:mm-are-equal},
  it also holds for every $\epsilon,\delta>0$.

  By taking $\delta \to 0$ for fixed $t,x,y$, and $\epsilon$ we thus get 
  \begin{align}\label{eq:balancing1}
   w(t,x)-w(t,y) \leq \frac{\epsilon}{4} + \frac{|x-y|^2}{\epsilon} 
   + c_w \, r^{-2\sigma}\Big(\frac{r^2}{\epsilon} 
   +\,\sqrt{\frac{r^2}{\epsilon}}+ \,\epsilon^{\hb/2}\Big),
  \end{align}
  where $c_w = 8CT\max\{1,\sqrt{\|w\|_\infty}\}$.
  To balance the second and the third terms in the parenthesis,
  we put  $r^2= \epsilon^ {\hb+1}$ for every $\epsilon\in(0,1)$.
  Since $\frac{(\hb+1)(1-2\sigma)-1}{2}\in(0,1)$ as $\hb\in(\frac{2\sigma}{1-2\sigma},1]$  and $2\sigma\in\big(0,\frac12\big)$,  this gives us by taking $C_1= \max\{ 1, 4c_w\}$
  \begin{align}\label{eq:xyeps}
   \begin{split}
    w(t,x)-w(t,y) 
    \leq C_1\bigg(\frac{|x-y|^2}{\epsilon} 
    + \epsilon^{\frac{(\hb+1)(1-2\sigma)-1}{2}}\bigg).
   \end{split}
  \end{align}
 For $|x-y|<1$, we balance terms again to find up to a constant that $\epsilon=|x-y|^{\bstr_1}$ where $ \bstr_1=\frac{4}{(\hb+1)(1-2\sigma) + 1} = \frac{4}{(\hb+2)(1-2\sigma) + 2\sigma}$.
  Notice that $\bstr_1\in (\frac{4}{3},2)$,
  since $\hb\in (\frac{2\sigma}{1-2\sigma},1]$ and  $2\sigma\in(0,\frac12)$.
  Then from \eqref{eq:xyeps} we obtain
  \begin{align}\label{eq:alpha}
   w(t,x)-w(t,y)  \leq 2 C_1 |x-y|^{2-\bstr_1},
  \end{align}
  which consequently holds for every~$(t,x,y)\in\Tb\times\X\times\X$ such that $|x-y|<1$ (cf.~\autoref{def:holder}).
 \end{step}
 \begin{step}\label{step:bootstrap}
  We ``bootstrap'' the argument of \autoref{step:alpha} to improve the H\"older ex\-po\-nent.
  By combining~\eqref{eq:square-estimate} and~\eqref{eq:alpha},
  after passing to the limit $\eta_k\to0$, we get
  \begin{align*}
   \frac{|x_{\epsilon,\delta}-y_{\epsilon,\delta}|^2}{\epsilon} 
   \leq  w(t_{\epsilon,\delta},x_{\epsilon,\delta})-w(t_{\epsilon,\delta},y_{\epsilon,\delta}) 
   \leq 2C_1 |x_{\epsilon,\delta}-y_{\epsilon,\delta}|^{2-\bstr_1}.
  \end{align*}
  It follows that $|x_{\epsilon,\delta}-y_{\epsilon,\delta}|^{\bstr_1}\leq 2 C_1\, \epsilon$.

  Now we go back to~\eqref{eq:mm-estimate} and follow the subsequent arguments, using the new bound.
  By taking $c_1 = 8CTC_1^{1/\bstr_1}$ (note that $C_1,\bstr_1\geq 1$) we obtain
   \begin{align}\label{eq:balancing2}
   w(t,x)-w(t,y) 
   \leq \frac{\epsilon}{4} + \frac{|x-y|^2}{\epsilon} 
   + c_1 \, r^{-2\sigma}\Big(\frac{r^2}{\epsilon} 
   + \epsilon^{\frac{1}{\bstr_1}-1}r+\epsilon^{\frac{\hb}{\bstr_1}}\Big). 
  \end{align}
    To balance the 
    terms in the parenthesis,
        for every $\epsilon\in (0,1)$ we put $r^{\bstr_1}=\epsilon^{\hb+\bstr_1-1}$
    (see \autoref{rem:optimal2} for an explanation).
    The dominant exponent of $\epsilon$ in~\eqref{eq:balancing2} is then $\frac{\hb-2\sigma(\hb+\bstr_1-1)}{\bstr_1}$ (and belongs to $(0,1)$ because $2\sigma+\frac{2\sigma}{1-2\sigma}
<\hb\leq 1<\bstr_1<2$), thus
  \begin{align*}
   w(t,x)-w(t,y) \leq C_2\bigg(\frac{|x-y|^2}{\epsilon} 
   + \epsilon^{\frac{\hb-2\sigma(\hb+\bstr_1-1)}{\bstr_1}}\bigg),
  \end{align*}
  where $C_2 = \max\{1, 4c_1\}$.
  Choosing $\epsilon=|x-y|^{\bstr_2}$
  for $\bstr_2=\frac{2\bstr_1}{(\hb+\bstr_1)(1-2\sigma) + 2\sigma}$ gives us
  \vspace{-\baselineskip}
  \begin{align}\label{eq:bootstrap}
   w(t,x)-w(t,y)  \leq 2 C_2 |x-y|^{2-\bstr_2}.
  \end{align}
  By repeating this procedure, we obtain  recursive formulas
  \begin{align}\label{eq:recursion}
  \left\{\begin{aligned}
   \bstr_0&=2,& \bstr_{n+1} &= \frac{2\bstr_n}{(\hb+\bstr_n)(1-2\sigma) + 2\sigma},\\
   C_0 &= \max\{1, \|w\|_{\infty}\},& C_{n+1} &= \max\{1, 32\,CTC_n^{1/\bstr_n}\},
   \end{aligned}\right.\qquad\text{for $n\in\N$.}
  \end{align} 
  Notice that $\hb+\bstr_{0}>\frac{2\sigma}{1-2\sigma}+2=\frac{2-2\sigma}{1-2\sigma}$.
  Now, assume  $\hb+\bstr_n>\frac{2-2\sigma}{1-2\sigma}$ for some $n\in\N$.
  Then,
  \begin{align*}
     \hb + \bstr_{n+1} 
      & = \frac{\hb(\hb+\bstr_n)(1-2\sigma) 
     + 2\hb\sigma+2\bstr_n}{(\hb+\bstr_n)(1-2\sigma) + 2\sigma} >\frac{2(\hb+\bstr_n)}{(\hb+\bstr_n)(1-2\sigma) + 2\sigma}
     \\& \ = 
     \frac{2}{(1-2\sigma) + \frac{2\sigma}{\hb+\bstr_n}} > \frac{2-2\sigma}{(1-2\sigma)(1-\sigma) + {(1-2\sigma)\sigma}}
     = \frac{2-2\sigma}{1-2\sigma}.
  \end{align*}
  By the principle of induction, we get $\hb+\bstr_n>\frac{2-2\sigma}{1-2\sigma}>2$ for every $n\in\N$.
  Then,
  \begin{align*}
   \frac{\bstr_{n+1}}{\bstr_n} 
   = \frac{2}{(\hb+\bstr_n)(1-2\sigma) + 2\sigma} <\frac{2}{2-2\sigma + 2\sigma} =1,
  \end{align*}
  i.e.~$\bstr_{n+1}<\bstr_n$.
  This also implies $\bstr_n\in(1,2]$, since  $\bstr_0=2$ and $2-\bstr_n<\hb\leq1$.
   Passing to the limit in \eqref{eq:recursion} we then find that  $\lim\limits_{n\to\infty}\bstr_n = \frac{2-2\sigma}{1-2\sigma}-\hb=\omega_\infty$.
  
  By \eqref{eq:recursion}, notice that $C_n\geq 1$ for every $n\in\N$.
  Moreover, if $32CT\leq 1$ and $C_{n_0}=1$ for some $n_0\in\N$, then $C_n=1$ for every $n\geq n_0$.
  In any other case, $C_{n+1} = 32CT C_n^{1/\bstr_n}$ for every $n\in\N$.
  Then
  \begin{align*}
   C_{n+1} = (32CT)^{\Sigma_n}\,C_0^{\Pi_n},
   \quad\text{where $\Pi_n = \prod_{k=1}^n\frac{1}{\bstr_k}$
   and $\Sigma_n = \Pi_n+\sum_{k=1}^n\frac{\Pi_n}{\Pi_k}$}.
  \end{align*}
  We observe that $\lim\limits\Pi_n=0$ because $\omega_n\geq \omega_\infty>1$ (since $\beta\leq1$ and $\sigma>0$) and  $\lim\limits\Sigma_n\leq \sum_{k=0}^\infty \frac1{\omega_\infty^k}= 1+\frac{1-2\sigma}{1-\hb(1-2\sigma)}<\infty$ since $\beta\leq1$.
  Either way, \mbox{$\lim\limits_{n\to\infty}C_n<\infty$}.
  
  By writing \eqref{eq:bootstrap} for every $n$ and then passing to the limit $n\to\infty$,
  we get $w\in \LC{\hb_0}$, where
  \begin{align*}
   \hb_0 = \lim_{n\to \infty} (2- \bstr_n) = \hb-\frac{2\sigma}{1-2\sigma}.
   \end{align*}
 \end{step}
 \begin{step}
 If the L\'evy measure $\nu$ is symmetric at the origin,
  the arguments in Steps~\ref{step:alpha} and \ref{step:bootstrap} lead to $w\in\LC{\widehat \hb_0}$,
  where \mbox{$\widehat \hb_0=\hb-\tfrac{2\sigma}{1-\sigma}$}
  (cf.~\autoref{rem:optimal1}\autoref{rem:optimal-symmetric1}), which also allows us to consider 
  {$\hb\in\big(\frac{2\sigma}{1-\sigma},1\big]$}.
  \qedhere
 \end{step}
 \end{proof}

 \begin{remark}\label{rem:optimal2}
  Our aim is to obtain the best H\"older regularity.
  The choice of scaling $r^2=\epsilon^{\hb+1}$ in~\eqref{eq:balancing1} is clearly optimal.
  When we repeat this argument in~\eqref{eq:balancing2},
  we want the \emph{lowest} of the three exponents to be the \emph{highest} possible.
  If $r=\epsilon^a$, then the exponents are
  \begin{align*}
   (2-2\sigma)a-1,\quad (1-2\sigma)a+1/\bstr_n-1,\quad -2\sigma a+\hb/\bstr_n,
  \end{align*}
  which are affine functions of $a$.
  The first two are increasing, and the third is decreasing,
  hence the optimal choice is at the intersection of either 1st and 3rd, or 2nd and 3rd lines,
  which corresponds to  
  $a= \max\big\{{\frac{\hb+\bstr_n-1}{\bstr_n}}, \,  {\frac{\hb+\bstr_n}{2\bstr_n}}\big\}$.
  We have $a= {\frac{\hb+\bstr_n-1}{\bstr_n}}$, since $\hb+\bstr_n>2$.
 \end{remark}

 \section*{Acknowledgements}
IC  was supported by the INSPIRE faculty fellowship (IFA22-MA187). ERJ received funding from the Research Council of Norway under Grant Agreement No. 325114 “IMod. Partial differential equations, statistics and data: An interdisciplinary approach to data-based modelling”.
 MK was supported by the Polish NCN grant 2016/23/B/ST1/00434 and Croatian Science Foundation grant IP-2018-01-2449.  
 The main part of the research behind this paper was conducted when  IC and MK were
 fellows of the ERCIM Alain Bensoussan Programme at NTNU.

 \bibliographystyle{siam}
 \let\OLDthebibliography\thebibliography
 \renewcommand\thebibliography[1]{
  \OLDthebibliography{#1}
  \setlength{\itemsep}{0pt plus 0.2ex}
 }
 \bibliography{MFG}
\end{document}